\documentclass[10pt]{article}

\usepackage{comment,url,algorithm,algorithmic,graphicx,subcaption,relsize}
\usepackage{amssymb,amsfonts,amsmath,amsthm,amscd,dsfont,mathrsfs,mathtools,microtype,nicefrac,pifont}
\usepackage{float,psfrag,epsfig,color,url,hyperref}
\usepackage{upgreek}
\usepackage[dvipsnames]{xcolor}
\usepackage{epstopdf,bbm,mathtools,enumitem}
\usepackage[toc,page]{appendix}
\usepackage[mathscr]{euscript}

\usepackage[top=1in, bottom=1in, left=1in, right=1in]{geometry}

\def\balign#1\ealign{\begin{align}#1\end{align}}
\def\baligns#1\ealigns{\begin{align*}#1\end{align*}}
\def\balignat#1\ealign{\begin{alignat}#1\end{alignat}}
\def\balignats#1\ealigns{\begin{alignat*}#1\end{alignat*}}
\def\bitemize#1\eitemize{\begin{itemize}#1\end{itemize}}
\def\benumerate#1\eenumerate{\begin{enumerate}#1\end{enumerate}}

\newenvironment{talign*}
 {\csname align*\endcsname}
 {\endalign}
\newenvironment{talign}
 {\csname align\endcsname}
 {\endalign}

\def\balignst#1\ealignst{\begin{talign*}#1\end{talign*}}
\def\balignt#1\ealignt{\begin{talign}#1\end{talign}}

\let\originalleft\left
\let\originalright\right
\renewcommand{\left}{\mathopen{}\mathclose\bgroup\originalleft}
\renewcommand{\right}{\aftergroup\egroup\originalright}

\def\tinycitep*#1{{\tiny\citep*{#1}}}
\def\tinycitealt*#1{{\tiny\citealt*{#1}}}
\def\tinycite*#1{{\tiny\cite*{#1}}}
\def\smallcitep*#1{{\scriptsize\citep*{#1}}}
\def\smallcitealt*#1{{\scriptsize\citealt*{#1}}}
\def\smallcite*#1{{\scriptsize\cite*{#1}}}

\def\<{\left\langle} 
\def\>{\right\rangle}

\DeclareSymbolFont{rsfs}{U}{rsfs}{m}{n}
\DeclareSymbolFontAlphabet{\mathscrsfs}{rsfs}

\providecommand{\dom}{\mathop\mathrm{dom}}

\ifdefined\nonewproofenvironments\else
\ifdefined\ispres\else
\newtheorem{theorem}{Theorem}
\newtheorem{lemma}[theorem]{Lemma}
\newtheorem{corollary}[theorem]{Corollary}

\renewenvironment{proof}{\noindent\textbf{Proof.}\hspace*{.3em}}{\qed \vspace{.1in}}
\newenvironment{proof-sketch}{\noindent\textbf{Proof Sketch}
  \hspace*{1em}}{\qed\bigskip\\}
\newenvironment{proof-idea}{\noindent\textbf{Proof Idea}
  \hspace*{1em}}{\qed\bigskip\\}
\newenvironment{proof-of-lemma}[1][{}]{\noindent\textbf{Proof of Lemma {#1}}
  \hspace*{1em}}{\qed\\}
\newenvironment{proof-of-theorem}[1][{}]{\noindent\textbf{Proof of Theorem {#1}}
  \hspace*{1em}}{\qed\\}
\newenvironment{proof-attempt}{\noindent\textbf{Proof Attempt}
  \hspace*{1em}}{\qed\bigskip\\}

\newenvironment{remark}{\noindent\textbf{Remark.}
  \hspace*{0em}}{\smallskip}

\fi

\newtheorem{proposition}[theorem]{Proposition}

\newtheorem{assumption}{Assumption}
\theoremstyle{definition}

\fi
\makeatletter
\@addtoreset{equation}{section}
\makeatother

\hypersetup{
  colorlinks,
  linkcolor={red!50!black},
  citecolor={blue!50!black},
  urlcolor={blue!80!black}
}

\usepackage{custom}
\usepackage[normalem]{ulem}

\makeatletter
\renewcommand{\paragraph}{%
  \@startsection{paragraph}{4}%
  {\z@}{1.25ex \@plus 1ex \@minus .2ex}{-1em}%
  {\normalfont\normalsize\bfseries}%
}
\makeatother

\begin{document}
\title{Towards a Theory of Non-Log-Concave Sampling: \\ First-Order Stationarity Guarantees for Langevin Monte Carlo}

 \author{\!\!\!\!\!
  Krishnakumar Balasubramanian\thanks{
Department of Statistics at University of California, Davis, \texttt{kbala@ucdavis.edu}
}
 \ \ \
 Sinho Chewi\thanks{
  Department of Mathematics at
  Massachusetts Institute of Technology, \texttt{schewi@mit.edu}
 }
 \ \ \
 Murat A. Erdogdu\thanks{
  Department of Computer Science at
  University of Toronto, and Vector Institute, \texttt{erdogdu@cs.toronto.edu}
 }
 \\
Adil Salim\thanks{
Microsoft Research, \texttt{salim@berkeley.edu}
}
 \ \ \
Matthew Zhang\thanks{
  Department of Computer Science at
  University of Toronto, and Vector Institute, \texttt{matthew.zhang@mail.utoronto.ca}
}
}

\maketitle

\begin{abstract}%
   For the task of sampling from a density $\pi \propto \exp(-V)$ on $\R^d$, where $V$ is possibly non-convex but $L$-gradient Lipschitz, we prove that averaged Langevin Monte Carlo outputs a sample with $\varepsilon$-relative Fisher information after $O( L^2 d^2/\varepsilon^2)$ iterations. This is the sampling analogue of complexity bounds for finding an $\varepsilon$-approximate first-order stationary points in non-convex optimization and therefore constitutes a first step towards the general theory of non-log-concave sampling. We discuss numerous extensions and applications of our result; in particular, it yields a new state-of-the-art guarantee for sampling from distributions which satisfy a Poincar\'e inequality.
\end{abstract}

\section{Introduction}

Consider the canonical task of sampling from a density $\pi \propto \exp(-V)$ on $\R^d$, given query access to the gradients of $V$. In the case where $V$ is strongly convex and smooth, this task is well-studied, with a number of works giving precise and non-asymptotic complexity bounds which scale polynomially in the problem parameters. In contrast, there are comparatively few works which study the case when $V$ is non-convex. In this work, we take a first step towards developing a general theory of non-log-concave sampling by formulating the sampling analogue of \emph{stationary point analysis}, which has been highly successful in the non-convex optimization~\cite{nesterov2018lectures}.

Classically, the \emph{Langevin diffusion}, the solution to the stochastic differential equation
\begin{align}\label{eq:langevin}
    \D z_t
    &= -\nabla V(z_t) \, \D t + \sqrt 2 \, \D B_t\,,
\end{align}
has $\pi$ as its unique stationary distribution and converges to it as $t\to\infty$ under mild conditions. Here, ${(B_t)}_{t\ge 0}$ is a standard $d$-dimensional Brownian motion. Discretizing this stochastic process with step size $h>0$ yields the standard \emph{Langevin Monte Carlo} (LMC) algorithm
\begin{align}\label{eq:lmc}\tag{\text{LMC}}
    x_{(k+1)h}
    &:= x_{kh} - h \, \nabla V(x_{kh}) + \sqrt 2 \, (B_{(k+1)h} - B_{kh})\,.
\end{align}
Several extensions of LMC have been considered in the literature. For instance, a stochastic gradient can be used as an estimate of the ``full'' gradient $\nabla V(x_{kh})$ at each iteration.

Although LMC and its extensions are ostensibly sampling algorithms, they find applications in optimization. Indeed, LMC and its extensions can be viewed as a variant of (stochastic) gradient descent in which Gaussian noise is explicitly injected in the (stochastic) gradient in each iteration. As explored, for example, in~\cite{raginskyrakhlintelgarsky2017sgld} and~\cite{jin2021nonconvex}, the presence of noise allows the iteration to escape local minima and allows for establishing global non-asymptotic convergence guarantees on well-behaved yet non-convex objectives. 

Perhaps surprisingly, the connection between optimization and sampling also goes in the other direction: the theory of optimization can be used to understand the performance of sampling algorithms. On a superficial level, this is anticipated because the Langevin diffusion~\eqref{eq:langevin} is simply a standard gradient flow to which a Brownian noise has been added. However, there is a much deeper connection, due to~\cite{jko1998}, which interprets the Langevin diffusion as an \emph{exact} gradient flow in the space of probability measures equipped with the geometry of optimal transport, where the objective functional is the Kullback--Leibler (KL) divergence $\KL(\cdot\mmid \pi)$. This perspective has spurred researchers to provide novel optimization-inspired analyses of sampling~\cite{bernton2018langevin, wibisono2018sampling, durmus2019analysis}.

For example, the Wasserstein gradient of $\KL(\cdot\mmid\pi)$ at $\mu$ is $\nabla \ln(\mu/\pi)$, and the calculation rules for gradient flows imply that if $\pi_t$ denotes the law of the Langevin diffusion~\eqref{eq:langevin} at time $t$, then $\partial_t \KL(\pi_t \mmid \pi) = -\E_{\pi_t}[\norm{\nabla \ln(\pi_t/\pi)}^2]$~\cite{ambrosio2008gradient, villani2009ot, santambrogio2015ot}. As this quantity is important in what follows, we explicitly write $\FI(\mu\mmid\pi) := \E_\mu[\norm{\nabla \ln(\mu/\pi)}^2]$ for the (relative) \emph{Fisher information} of $\mu$ w.r.t.\ $\pi$. If $V$ is convex (resp.\ strongly convex), then it turns out that the objective functional $\KL(\cdot\mmid \pi)$ is convex (resp.\ strongly convex) in the Wasserstein geometry, which in turn implies that $\KL(\pi_t\mmid\pi)$ decays to zero at the rate $O(1/t)$ (resp.\ exponentially fast).

In the case when $V$ is non-convex, however, less is known. Of course, just like non-convex optimization, it is in general impossible to obtain polynomial sampling guarantees for non-log-concave distributions. Recently,~\cite{vempala2019ulaisoperimetry,chewietal2021lmcpoincare, ma2021there} study tractable cases of non-log-concave sampling in which the target $\pi$ satisfies a \emph{functional inequality}, such as the \emph{log-Sobolev inequality} (LSI). Indeed, if LSI holds, then $\FI(\mu \mmid \pi) \gtrsim \KL(\mu \mmid \pi)$ for all $\mu$. In light of the Wasserstein calculus described above, this is the analogue of the \emph{gradient domination} condition (or \emph{Polyak--\L{}ojasiewicz inequality}) in non-convex optimization: $\|\nabla V(x)\|^2 \gtrsim V(x) - \min V$ \cite{lojasiewicz1963topological, polyak1963gradient,kariminutinischmidt2016pl}. Furthermore, \cite{durmus2017nonasymptotic, cheng2018sharp, li2019stochastic, majka2020nonasymptotic, erdogdu2021convergence, he2022heavy} study tractable classes of non-log-concave sampling based on certain tail-growth conditions. However, the assumptions made in all the above works are very far from capturing the breadth of non-log-concave sampling.

Instead, in general non-convex optimization, the standard approach is to prove convergence to a \emph{stationary point} of the objective function, or from a more quantitative perspective, to determine the complexity of obtaining a point $x$ satisfying $\norm{\nabla V(x)}^2 \le \varepsilon$. This complexity is typically $O(1/\varepsilon)$~\cite{nesterov2018lectures}. Following this paradigm, we propose to use the Fisher information as the sampling analogue of the squared norm of the gradient. Our main result (Theorem~\ref{thm:main}) establishes that under the sole assumption that $\nabla V$
 is $L$-gradient Lipschitz, an averaged version of the LMC algorithm~\eqref{eq:lmc} outputs a sample whose law $\mu$ satisfies $\FI(\mu\mmid\pi) \le\varepsilon$ after $O(L^2 d^2/\varepsilon^2)$ iterations.  Intuitively, the Fisher information captures the rapid local mixing of the Langevin diffusion near modes of the distribution $\pi$, while ignoring the metastability effects which occur between the modes~\cite{bovier2002metastability, bovier2004metastabilitya,bovier2004metastabilityb}. We give an illustrative example in Section~\ref{scn:example} which expands upon this intuition.
 
 \subsection{Paper organization and contributions}
 The rest of the paper is organized as follows. In Section~\ref{scn:example}, we provide intuitions on Fisher information guarantees in sampling. In Section~\ref{scn:preliminaries}, we formally define the Fisher information, and in Section~\ref{scn:main}, we state our main result in Theorem~\ref{thm:main}. 
 In Section~\ref{scn:applications}, we consider applications of our main result: 
 \begin{itemize}
     \item We show the weak convergence of averaged LMC with decaying step size (Section~\ref{scn:asymptotic}).
     \item  We provide new sampling guarantees in total variation distance under Poincaré inequality (Section~\ref{scn:poincare_smooth}). These guarantees are competitive with very recent results by~\cite{chewietal2021lmcpoincare} (in fact, our dimension dependence is substantially better). 
\item We show an accelerated convergence result for LMC when the Hessian of the potential is also Lipschitz and when the potential satisfies a polynomial tail growth condition (Section \ref{scn:hessian}).
     \end{itemize}
     
     In Section~\ref{scn:extensions}, we consider extensions of LMC involving stochastic gradients.
     \begin{itemize}
        \item First, we consider the general case where the stochastic gradients admit a bounded bias and a bounded variance (Section~\ref{scn:general_stochastic}).
        \item As a corollary of this general result, we obtain convergence guarantees for LMC in the case where $V$ is only \emph{weakly} smooth, i.e., $\nabla V$ is H\"older continuous (Section~\ref{scn:weakly_smooth}). We employ the Gaussian smoothing technique to obtain this corollary. It implies new sampling guarantees in total variation distance under a Poincaré inequality and weak smoothness.
        \item We obtain convergence guarantees for LMC in the case where $V$ is a finite sum and the stochastic gradients are defined from mini-batches. In this case, the stochastic gradients have zero bias but unbounded variance. We employ the variance reduction technique (Section~\ref{scn:vr}).
     \end{itemize}

 Finally, we conclude with open directions in Section~\ref{scn:conclusion}.

\section{Interpretation of approximate first-order stationarity in sampling}\label{scn:example}

Intriguingly, unlike the situation in non-convex optimization, in sampling there are no ``spurious stationary points'': if $\mu$ and $\pi$ have positive and smooth densities and $\FI(\mu\mmid \pi) = 0$, then $\mu = \pi$. However, for $\varepsilon > 0$, it may be unclear what the guarantee $\FI(\mu\mmid \pi) \le \varepsilon$ entails.
In this section, we give an example illustrating what conclusions may be drawn from a bound on the Fisher information, which helps to better interpret our result in the next sections.

\begin{figure}\label{fig:mupi}
  \vspace{-20pt}
  \begin{center}
    \includegraphics[width=0.5\textwidth]{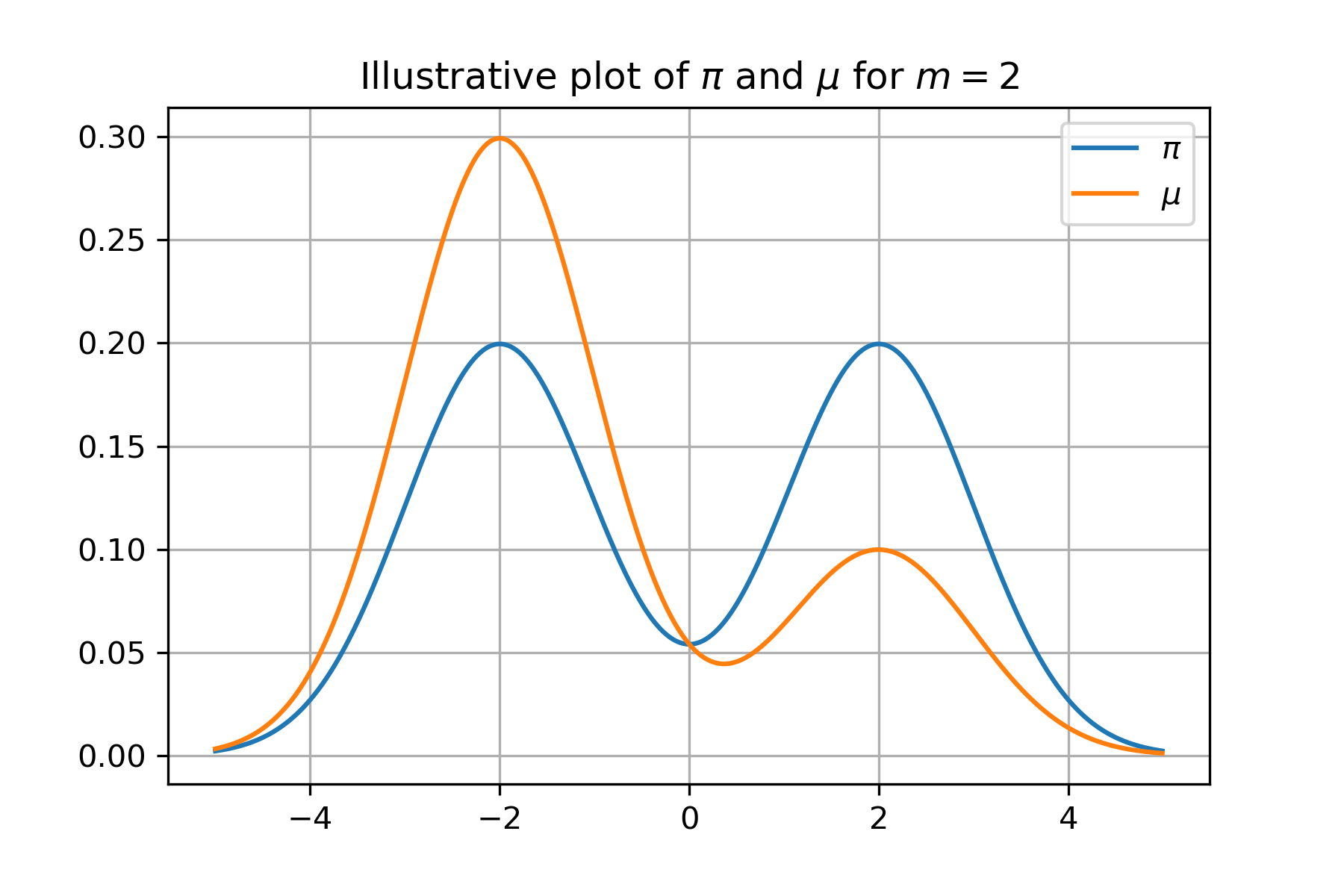}
  \end{center}
  \vspace{-20pt}
  \vspace{-10pt}
\end{figure}
Consider a mixture of two Gaussians in one dimension as the target distribution:
\begin{align*}
    \pi
    &= {\frac{1}{2}\, \underbrace{\normal(-m, 1)}_{\pi_-}} + {\frac{1}{2}\, \underbrace{\normal(+m, 1)}_{\pi_+}}\,,
\end{align*}
where $m \gg 0$. Also, consider a mixture of two Gaussians with different weights:
\begin{align*}
    \mu
    &:= \frac{3}{4} \, \pi_- + \frac{1}{4} \, \pi_+\,.
\end{align*}
An illustrative plot of $\pi$ and $\mu$ is provided for the sake of easier visualization. In the appendix, we will prove the following.

\begin{proposition}\label{prop:example}
Let $\pi$ and $\mu$ be as defined above. For all $m\ge 1/80$, it holds that
    \begin{align*}
        \norm{\mu-\pi}_{\rm TV}
        \ge \frac{1}{800} > 0\,.
    \end{align*}
    On the other hand,
    \begin{align*}
        \FI(\mu \mmid \pi)
        &\le 4m^2 \exp\bigl(-\frac{m^2}{2}\bigr)
        \to 0 \qquad\text{as}~m\to\infty\,.
    \end{align*}
\end{proposition}

In the next section, we will show that averaged LMC can drive the Fisher information to zero at a polynomial rate. For large $m$, the measure $\mu$ has small Fisher information with respect to $\pi$, so $\mu$ serves as a model for the kind of distribution that averaged LMC can reach. We can draw a few conclusions:
\begin{enumerate}
    \item Although the Fisher information $\FI(\mu\mmid\pi)$ is very small, the total variation distance remains bounded away from zero. This shows that a Fisher information guarantee does \emph{not} ensure fast convergence of averaged LMC in other metrics without further assumptions (anyway, polynomial guarantees for non-log-concave sampling in other metrics are impossible in general). 
    \item Here, $\mu$ \emph{locally} captures the correct shape of $\pi$ at the two modes. On the other hand, $\mu$ has different mixing weights than $\pi$, which means that $\mu$ is \emph{globally} different from $\pi$. Since $\FI(\mu\mmid \pi)$ is small for this example, it shows that the Fisher information is not sensitive to the latter effect. Hence, our Fisher information guarantee for averaged LMC captures the fact that the algorithm rapidly gets the local structure of $\pi$ correct.
    \item After a few steps of LMC started at the distribution $\frac{3}{4} \, \delta_{-m} + \frac{1}{4} \, \delta_{+m}$, the algorithm arrives at a measure which closely resembles $\mu$, rather than the true stationary measure $\pi$. Indeed, the iterates of LMC do not need to jump from one mode to another to approximate $\mu$. This jumping takes an exponentially long time and is the main barrier to the mixing of LMC, but it is necessary for LMC to learn the global mixing weights\textemdash{}this is known as the \emph{metastability} phenomenon~\cite{bovier2002metastability, bovier2004metastabilitya,bovier2004metastabilityb}. Our analysis provides a convenient way to quantify this effect.
\end{enumerate}

\begin{remark}
In the context of Bayesian inference, the choice of relative Fisher information metric between the prior and the exact posterior distribution has been proposed by~\cite{walker2016bayesian, holmes2017assigning,shao2019bayesian}, as a measure of robustness of the overall inferential procedure. In this regard, our results provide a computational angle to this paradigm: in practice we rarely have access to the exact posterior distribution. Our results algorithmically quantify the distance (in relative Fisher information) between the posterior distribution obtained after a certain number of iterations of LMC and the exact posterior.
\end{remark}

\section{Preliminaries}\label{scn:preliminaries}

Throughout the paper, we assume that the potential $V : \R^d\to\R$ is a smooth (i.e., twice continuously differentiable) function such that $\int \exp(-V) < \infty$. The target distribution $\pi \propto \exp(-V)$ is therefore well-defined.

For a probability measure $\mu$ with a smooth density, we can define the \emph{Fisher information} of $\mu$ relative to $\pi$ via $\FI(\mu\mmid\pi) := \E_\mu[\norm{\nabla \ln(\mu/\pi)}^2]$. To extend this definition to other probability measures, we recall from Markov semigroup theory~\cite{bakrygentilledoux2014} that we associate with the Langevin diffusion~\eqref{eq:langevin} a \emph{Dirichlet energy} $f \mapsto \ms E(f)$ which maps a subspace $\dom \ms E \subseteq L^2(\pi)$ to $\R_+$. If $f$ is smooth and compactly supported, then $f \in \dom \ms E$ and the Dirichlet energy has the explicit expression $\ms E(f) = \E_\pi[\norm{\nabla f}^2]$.
The Fisher information is defined from the Dirichlet energy as follows. For an arbitrary probability measure $\mu$, set
\begin{align*}
    \FI(\mu\mmid \pi)
    &:= \begin{cases} 4 \, \ms E(\sqrt f)\,, & \text{if}~f := \frac{\D \mu}{\D \pi}~\text{exists and}~\sqrt f \in \dom \ms E\,, \\
    +\infty\,, & \text{otherwise}\,.
    \end{cases}
\end{align*}
In particular, if $f = \frac{\D \mu}{\D \pi}$ is positive and smooth, one can check that 
\begin{equation*}
    \FI(\mu\mmid \pi)
    = \int \norm{\nabla \ln(f)}^2 \, \D \mu\,, \qquad \text{or} \qquad     \FI(\mu\mmid \pi)
    = \int \frac{\norm{\nabla f}^2}{f} \, \D \pi\,.
\end{equation*}
Using the convexity of $(a,b) \mapsto \norm a^2/b$ on $\R^d\times\R_+$, the latter formula implies that the Fisher information $\mu \mapsto \FI(\mu\mmid \pi)$ is convex in the classical sense on the space of probability measures. Besides, the Fisher information is also lower semicontinuous in its first argument with respect to the weak topology of measures~\cite[Appendix B]{wu2000largedeviations}.

\section{Main result}\label{scn:main}
Recall that the LMC algorithm is given by
\begin{equation*}
    x_{(k+1)h}
    := x_{kh} - h \, \nabla V(x_{kh}) + \sqrt 2 \, (B_{(k+1)h} - B_{kh})\,.
\end{equation*}
Our main result is stated for the following continuous interpolation of LMC:
\begin{align}\label{eq:interpolated_lmc}
    x_t
    &:= x_{kh} - (t-kh) \, \nabla V(x_{kh}) + \sqrt 2 \, (B_t - B_{kh}) \qquad\text{for}~t\in [kh, (k+1)h]\,.
\end{align}
We write $\mu_t$ for the law of $x_t$.

\begin{assumption}\label{as:grad_smooth}
The gradient of $V$ is $L$-Lipschitz continuous: $\norm{\nabla V(x_1) - \nabla V(x_2)} \leq L \, \norm{x_1 - x_2}$, for all $x_1, x_2 \in \mathbb{R}^d$ and for some $L > 0$.
\end{assumption}

\begin{theorem}\label{thm:main}
    Let ${(\mu_t)}_{t\ge 0}$ denote the law of the interpolation~\eqref{eq:interpolated_lmc} of LMC, and let the potential $V$ satisfy Assumption \ref{as:grad_smooth}.
    Then, for any step size $h \in (0, \frac{1}{6L})$, it holds that
    \begin{align*}
        \frac{1}{Nh} \int_0^{Nh} \FI(\mu_t \mmid \pi) \, \D t
        &\le \frac{2\KL(\mu_0\mmid\pi)}{Nh} + 8L^2 dh\,.
    \end{align*}
    In particular, if $\KL(\mu_0\mmid \pi)\le K_0$ and we choose $h = \sqrt K_0/(2L\sqrt{dN})$, then for $N \ge 9K_0/d$,
    \begin{align*}
        \frac{1}{Nh} \int_0^{Nh} \FI(\mu_t \mmid \pi) \, \D t
        &\le \frac{8L\sqrt{dK_0}}{\sqrt N}\,.
    \end{align*}
\end{theorem}

By the convexity of the Fisher information, it follows that the averaged distribution $\bar\mu_{Nh} := {(Nh)}^{-1} \int \mu_t \, \D t$ satisfies $\FI(\bar\mu_{Nh} \mmid \pi) \le 8L \sqrt{dK_0/N}$ as well. Also, it is possible to output a sample from $\bar\mu_{Nh}$, as follows:
\begin{enumerate}
    \item Pick a time $t \in [0,Nh]$ uniformly at random.
    \item Let $k$ be the largest integer such that $kh \le t$, and let $x_{kh}$ be the iterate of LMC at time $kh$. Then, perform a partial LMC update for time $t-kh$, i.e.\ set
    \begin{align*}
        x_t
        &:= x_{kh} - (t-kh) \, \nabla V(x_{kh}) + \sqrt 2 \, (B_t - B_{kh})\,.
    \end{align*}
    Then, $x_t$ is a sample from $\bar \mu_{Nh}$. 
    Note that it is possible to sample the Brownian increments exactly as long as one can draw standard Gaussian vectors.
\end{enumerate}
\begin{remark}
Since we can usually take $K_0$ to be of order $d$, see \textit{e.g.}~\cite[Lemma 1]{vempala2019ulaisoperimetry} or~\cite[Appendix A]{chewietal2021lmcpoincare}, in order for averaged LMC to reach $\varepsilon$ accuracy in terms of the Fisher information w.r.t.\ the target, the iteration complexity is $O(L^2 d^2/\varepsilon^2)$. 
\end{remark}
\section{Applications}\label{scn:applications}

\subsection{Asymptotic convergence of averaged LMC with vanishing step size}\label{scn:asymptotic}
Our main result immediately implies asymptotic convergence of averaged LMC with decreasing step size under very general conditions.
Let ${(h_k)}_{k=1}^\infty$ be a sequence of positive step sizes such that 
\begin{align}\label{eq:conditionsonh}
\sum_{k=1}^\infty h_k = \infty~~~\text{and}~~~\sum_{k=1}^\infty h_k^2 < \infty\,.
\end{align}
Write $\tau_n := \sum_{k=1}^n h_k$, and denote by $\bar \mu_{\tau_n} := \tau_n^{-1} \int_0^{\tau_n} \mu_t\,\D t$, where $\mu_t$ is the law of $x_t$ defined by
\begin{align*}
    x_t
    &= x_{\tau_{n-1}} - (t-\tau_{n-1})\,\nabla V(x_{\tau_{n-1}}) + \sqrt 2 \, (B_t - B_{\tau_{n-1}})\,, \qquad t \in [\tau_{n-1},\tau_n]\,.
\end{align*}
Then, we have the following convergence result. 
\begin{theorem}\label{thm:asymptotic}
     Let ${(\mu_t)}_{t\ge 0}$ denote the law of the interpolation~\eqref{eq:interpolated_lmc} of LMC, and let the potential $V$ satisfy Assumption \ref{as:grad_smooth}. Suppose that LMC is initialized at a measure $\mu_0$ with $\KL(\mu_0 \mmid \pi) < \infty$ and that the step size sequence ${(h_k)}_{k=1}^\infty$ satisfy $h_k \in (0, \frac{1}{6L})$ for every $k$, as well as the conditions in~\eqref{eq:conditionsonh}.
    Then, $\bar \mu_{\tau_n} \to \pi$ weakly.
\end{theorem}

While it might be possible to prove the weak convergence of LMC using other techniques, for example, the ordinary differential equation method from the stochastic approximation literature~\cite{kus-yin-livre03} or general results on the analysis of Markov chains~\cite{bakrygentilledoux2014,douc2018markov}, we emphasize that Theorem~\ref{thm:asymptotic} follows immediately from our main result in Theorem~\ref{thm:main} and the property that $\FI(\mu\mmid \pi) = 0$ implies $\mu = \pi$. To the best of our knowledge, explicit results available in the literature on the weak convergence of LMC~\cite{lamberton2002recursive,pages2012ergodic} require Lyapunov-type conditions. In comparison, Theorem~\ref{thm:asymptotic} holds just under the Lipschitz gradient assumption on the potential $V$.


\subsection{New sampling guarantees under a Poincar\'e inequality}\label{scn:poincare_smooth}

In this section, we show that if we additionally assume that $\pi$ satisfies a Poincar\'e inequality, then we obtain sampling guarantees in total variation distance as a corollary of our main theorem. Surprisingly, the rates we obtain in this manner are competitive with (and arguably better than) the state-of-the-art results for LMC, for these classes of target distributions. To present our result, we recall the following transportation inequality.

\begin{lemma}[{\cite[Theorem 3.1]{guillinetal2009transportinfo}}]\label{lem:transport_ineq}
    Suppose that $\pi$ satisfies a Poincar\'e inequality: for all smooth compactly supported functions $f : \R^d\to\R$,
    \begin{align}\label{eq:pi}\tag{\text{PI}}
        \var_\pi f
        &\le C_{\msf{PI}} \E_\pi[\norm{\nabla f}^2]\,.
    \end{align}
    Then, for all probability measures $\mu$,
    \begin{align*}
        \norm{\mu-\pi}_{\rm TV}^2
        &\le 4C_{\msf{PI}} \FI(\mu\mmid\pi)\,.
    \end{align*}
\end{lemma}
When combined with Theorem~\ref{thm:main}, we immediately obtain the following corollary.

\begin{corollary}\label{cor:poincare}
     Let ${(\mu_t)}_{t\ge 0}$ denote the law of the interpolation~\eqref{eq:interpolated_lmc} of LMC, and let the potential $V$ satisfy Assumption \ref{as:grad_smooth}.
    If $\KL(\mu_0\mmid \pi)\le K_0$ and we choose $h = \sqrt K_0/(2L\sqrt{dN})$, then for $N \ge 9K_0/d$ and $\bar\mu_{Nh} := {(Nh)}^{-1} \int_0^{Nh} \mu_t \, \D t$,
    \begin{align*}
        \norm{\bar\mu_{Nh} - \pi}_{\rm TV}^2
        &\le \frac{32C_{\msf{PI}} L\sqrt{dK_0}}{\sqrt N}\,.
    \end{align*}
\end{corollary}

\begin{remark}
If $K_0 = O(d)$, it implies an iteration complexity of $O(C_{\msf{PI}}^2 L^2 d^2/\varepsilon^2)$ to output a sample whose squared total variation distance to $\pi$ is at most $\varepsilon$. We are aware of only one other work which provides sampling guarantees for smooth potentials satisfying a Poincar\'e inequality: the recent result of~\cite[Theorem 7]{chewietal2021lmcpoincare} yields an iteration complexity of $\widetilde O(C_{\msf{PI}}^2 L^2 d^3/\varepsilon)$ for LMC (without averaging). Our result has worse dependence on the inverse accuracy, but better dependence on the dimension.
\end{remark}

Using \emph{Gaussian smoothing} we can extend Corollary~\ref{cor:poincare} to the case when $\nabla V$ is only H\"older continuous rather than Lipschitz continuous. This requires extending the Theorem~\ref{thm:main} to accommodate stochastic gradients, and hence it is deferred to Section~\ref{scn:weakly_smooth}.

\subsection{Hessian smoothness}\label{scn:hessian}
While our main results were obtained under Lipschitz smoothness of the gradient of the potential, prior analyses of Langevin algorithms~\cite{dalalyan2019user, mou2019improved} suggest that convergence rates are accelerated under a smoothness assumption on the Hessian.


\begin{assumption}\label{as:hess_smooth}
The Hessian of $V$ is $M$-Lipschitz: $\norm{\nabla^2 V(x_1) - \nabla^2 V(x_2)}_{\rm op} \leq M \, \norm{x_1 - x_2}$, for all $x_1, x_2 \in \mathbb{R}^d$ and for some $M > 0$.
\end{assumption}
Additionally, we require an upper bound on the order of growth of the function.
\begin{assumption}\label{as:hess_growth}
    There exist parameters $\constG \in [0,2]$, $0 \leq \constZ \leq \constG/2$, and constants $a,b,\constL > 0$ such that for all $x\in\R^d$,
        \begin{align}
            \label{eq:hess_growth}
                \langle x, \nabla V(x)\rangle \geq a\,\|x\|^\constG - b\ \ \text{ and }\ \ \|\nabla V(x)\| \leq \constL\, (1+ \|x\|^\constZ)\,.
        \end{align}
\end{assumption}
Note that assuming $\gamma > 2$ would contradict Lipschitz smoothness of the gradient.
The final condition allows for any polynomial tail for the potential, thus covers a significantly more general setting than the dissipativity assumption appearing in \cite{raginskyrakhlintelgarsky2017sgld,erdogdu2018global} ($\constG=2$) and the growth considered in \cite{chewietal2021lmcpoincare,erdogdu2021convergence} ($\constG \geq 1$). In the special case where $\constZ = \constG = 0$, it is equivalent to the gradient $\norm{\nabla V}$ being uniformly bounded, i.e.\ $V$ itself is Lipschitz.
The growth condition is used to establish new moment bounds for the iterates of LMC (Proposition~\ref{prop:moment-bounds}), which are key for discretization analysis.

In the following theorem, we assume for simplicity that $\constA = 1$ (which can be achieved by rescaling the potential).

\begin{theorem}
\label{thm:hess_smooth}
   Let ${(\mu_t)}_{t\ge 0}$ denote the law of the interpolation~\eqref{eq:interpolated_lmc} of LMC, and let the potential $V$ satisfy Assumptions \ref{as:grad_smooth}, \ref{as:hess_smooth}, and \ref{as:hess_growth}. Assume $\constA = 1$ and that the initialization is chosen with
   $\E[\norm{x_0}^4] \leq \sigma^2 d^2$ for some $\sigma \ge 3$. Define the parameter $\kappa := 1 \vee L \vee  M^{2/3} \vee (M^{1/3} \constL^{2/3})$. If the step size is chosen to satisfy $0 < h \lesssim \frac{1}{L}\wedge \frac{1}{\constL^2} \wedge 1$, then
   \begin{align*}
       \frac{1}{Nh} \int_0^{Nh} \FI(\mu_t \mmid \pi) \, \D t &\lesssim \frac{\KL(\mu_0 \mmid \pi)}{Nh} + \kappa^3 d^2 h^2 + \kappa^6 \,{(b+\sigma d)}^3\, Nh^5\,.
   \end{align*}
   If $\msf{KL}(\mu_0 \mmid \pi) \leq K_0$, and $h \asymp \frac{K_0^{1/3}}{\kappa \,{(\constB + \sigma d)}^{2/3} \, N^{1/3}}$ while $N \gtrsim \frac{K_0 \, (L^3 \vee m^6)}{{\kappa^3 (b+\sigma d)}^{2}}$, the following bound holds:
\begin{align*}
    \frac{1}{Nh} \int_0^{Nh} \FI(\mu_t \mmid \pi) \, \D t &\lesssim
     \Bigl( {(b+\sigma d)}^{2/3} K_0^{2/3} + \frac{K_0^{5/3}}{{(b+\sigma d)}^{1/3}} \Bigr)\, \frac{\kappa}{N^{2/3}}\,.
\end{align*}
\end{theorem}

Note that the bound is independent of the growth exponents $\gamma, \xi$ found in Assumption~\ref{as:hess_growth}.

\begin{remark}
When $b, K_0 = O(d)$, the iteration complexity implied by this result is $O(d^2/\varepsilon^{3/2})$, which should be compared to the complexity of $O(d^2/\varepsilon^2)$ in Theorem~\ref{thm:main}.
As in Corollary~\ref{cor:poincare}, we can combine this result with Lemma~\ref{lem:transport_ineq} to obtain the complexity $O(d^2/\varepsilon^{3/2})$ in squared total variation distance under the additional assumption of a Poincar\'e inequality on the target.
\end{remark}



\section{Extension to stochastic gradients}\label{scn:extensions}

\subsection{General result}\label{scn:general_stochastic}

We proceed to prove a more general result in which the gradient term in \eqref{eq:lmc} is replaced by a stochastic gradient. More precisely,  we use a stochastic estimate $G(x_{kh}, \zeta_k)$ of the gradient $\nabla V(x_{kh})$, where the random variables ${(\zeta_k)}_{k\in\N}$ representing the external randomness are i.i.d.\ and independent of all other random variables. Thus, we obtain \emph{stochastic gradient Langevin Monte Carlo} (SG-LMC):
\begin{align}\label{eq:stoch-lmc} \tag{SG-LMC}
    x_{(k+1)h}
    &:= x_{kh} - h \, G(x_{kh}, \zeta_k) + \sqrt 2 \, (B_{(k+1)h} - B_{kh})\,.
\end{align}

\begin{assumption}[Regularity of the stochastic gradient]\label{assump:sgorcle}
 Let $\E_\zeta G(y, \zeta) = \nabla \hat V(y)$ for some function $\hat V :\R^d\to\R$. The stochastic gradient $G(x,\zeta) \in \mathbb{R}^d$ satisfies:
\begin{itemize}
    \item Smoothness of the expected stochastic gradient: $\nabla\hat V$ is $\hat L$-Lipschitz.
    \item Bias bound: $\| \nabla \hat V(x) - \nabla V(x)\|^2 \leq \Cbias$ for all $x\in\R^d$.
    \item Variance bound: $\E_\zeta[\| G(x,\zeta) -\nabla \hat V(x)  \|^2] \leq \Cvar$ for all $x\in\R^d$.
\end{itemize}
\end{assumption}
We present the following theorem regarding the convergence:
\begin{theorem}\label{thm:stochastic_setting}
 Let ${(\mu_t)}_{t\ge 0}$ denote the law of the interpolation of~\eqref{eq:stoch-lmc}. Assume that the stochastic oracle satisfies Assumption~\ref{assump:sgorcle}. Then, for all $h \in (0, \frac{1}{14\hat L})$, we have
     \begin{align*}
  \frac{1}{Nh} \int_0^{Nh} \FI(\mu_t \mmid \pi) \, \D t
  &\leq \frac{2 \, \msf{KL}(\mu_0 \mmid \pi)}{Nh} + 16\hat L^2 dh + 8 \, (\Cbias + \Cvar)\,.
    \end{align*}
     In particular, if $\KL(\mu_0\mmid \pi)\le K_0$ and we choose $h = \sqrt K_0/(\hat L\sqrt{8dN})$, then for $N \ge 25K_0/d$,
    \begin{align*}
        \frac{1}{Nh} \int_0^{Nh} \FI(\mu_t \mmid \pi) \, \D t
        &\le \frac{16\hat L\sqrt{2dK_0}}{\sqrt N} + 8 \, (\Cbias + \Cvar)\,.
    \end{align*}
\end{theorem}

This generic result allows us to use \textit{biased} stochastic gradients. In particular, it can be applied to LMC with Gaussian smoothing.

\subsection{Extension to non-smooth potentials via Gaussian smoothing}\label{scn:weakly_smooth}

In this section, we extend our main theorem (Theorem~\ref{thm:main}) to the case when $\nabla V$ is assumed to be H\"older continuous.
\begin{assumption}\label{as:holder}
The gradient of $V$ is H\"older continuous of exponent $s\in (0, 1]$:
\begin{align*}
    \norm{\nabla V(x_1) - \nabla V(x_2)} \leq L \, \norm{x_1 - x_2}^s
\end{align*}
for all $x_1, x_2 \in \mathbb{R}^d$ and for some $L > 0$.
\end{assumption}
We consider the Gaussian smoothing LMC algorithm analyzed in~\cite{chatterji2020langevin}:
\begin{align}
\label{eq:gsn_smoothing}
    x_{(k+1)h} &= x_{kh} - h\, \nabla V(x_{kh} + \eta \zeta_k) + \sqrt{2} \,(B_{(k+1)h} - B_{kh})\,,
\end{align}
where $\eta > 0$ is a smoothing parameter and ${(\zeta_k)}_{k\in\N}$ are i.i.d.\ standard Gaussian random variables on $\R^d$ independent from $x_0$ and ${(B_t)}_{t\ge 0}$. We see that this iteration is a special case of~\eqref{eq:stoch-lmc} with stochastic gradient  given by $G(x_{kh},\zeta_k) = \nabla V(x_{kh} + \eta \zeta_k)$. The expected stochastic gradient is $\E G(x,\zeta) = \nabla \hat V(x)$, where $\hat V(x) = \E V(x+\eta\zeta)$ and $\zeta \sim \normal(0, I_d)$.

From~\cite[Lemma 2.2 and Lemma 3.1]{chatterji2020langevin}, $\nabla \hat V$ satisfies the first and third conditions of Assumption~\ref{assump:sgorcle} with
\begin{align*}
    \hat L \le\frac{L d^{(1-s)/2}}{\eta^{1-s}}\,,\qquad \Cvar \le 4L^2 d^s \eta^{2s}\,.
\end{align*}
To control the bias, we extend the result of~\cite{nesterov2017random}. 

\begin{lemma}\label{lem:gaussian_smoothing_bias}
    The Gaussian smoothed potential $\hat V$ with smoothing parameter $\mu$ satisfies the second condition of Assumption~\ref{assump:sgorcle} with
    \begin{align*}
        \Cbias
        &\lesssim L^2 d^{2+s} \eta^{2s}\,.
    \end{align*}
\end{lemma}
From the lemma, we see that the bias dominates: $\Cbias \gtrsim \Cvar$.
We obtain the following corollary.

\begin{corollary}
    Let ${(\mu_t)}_{t\ge 0}$ denote the law of the interpolation~\eqref{eq:interpolated_lmc} of Gaussian smoothed LMC~\eqref{eq:gsn_smoothing}, and let the potential $V$ satisfy Assumption \ref{as:holder}. If $\KL(\mu_0\mmid \pi)\le K_0$, we choose the smoothing to be $\eta\asymp \varepsilon^{1/(2s)}/(L^{1/s} d^{(2+s)/(2s)})$ (where the $\asymp$ hides an absolute constant), and we choose the step size $h$ as in Theorem~\ref{thm:stochastic_setting}, then the averaged law $\bar\mu_{Nh} \le {(Nh)}^{-1} \int_0^{Nh}\mu_t \, \D t$ satisfies $\FI(\bar\mu_{Nh} \mmid \pi) \le \varepsilon$, provided that the number of iterations is
    \begin{align*}
        N
        &\gtrsim \frac{K_0 L^{2/s} d^{(2+s-2s^2)/s}}{\varepsilon^{(1+s)/s}}\,.
    \end{align*}
\end{corollary}
\begin{proof}
    Apply Theorem~\ref{thm:stochastic_setting}.
\end{proof}

When $K_0 = O(d)$, the iteration complexity is $O(L^{2/s} d^{2\,(1+s-s^2)/s}/\varepsilon^{(1+s)/s})$.
As in Section~\ref{scn:poincare_smooth}, this result can be combined with a Poincar\'e inequality to yield a convergence result in total variation distance. However, there is a better approach in this case. It turns out that to reduce the variance of the stochastic gradients in Gaussian smoothing, it is advantageous to consider mini-batching: for $B \in \N^+$, we consider
\begin{align}\label{eq:minibatch_gs}
    x_{(k+1)h} &= x_{kh} - \frac{h}{B} \sum_{\ell=1}^B \nabla V(x_{kh} + \eta \zeta_{k,\ell}) + \sqrt{2} \,(B_{(k+1)h} - B_{kh})\,,
\end{align}
where ${(\zeta_{k,\ell})}_{k,\ell\in\N}$ is a family of i.i.d.\ standard Gaussians on $\R^d$ independent of $x_0$ and ${(B_t)}_{t\ge 0}$. 


\begin{corollary}\label{cor:poincare_weakly_smooth}
    Let ${(\mu_t)}_{t\ge 0}$ denote the law of the interpolation of the Gaussian smoothed LMC with mini-batching~\eqref{eq:minibatch_gs}, and let the potential $V$ satisfy Assumption \ref{as:holder}. Assume moreover that $\pi$ satisfies the Poincar\'e inequality~\eqref{eq:pi} with constant $C_{\msf{PI}}$.
    If $\KL(\mu_0\mmid \pi)\le K_0$, we choose the smoothing $\eta$ appropriately (see~\eqref{eq:batch_optimal_step_size}), and we choose the step size $h$ as in Theorem~\ref{thm:stochastic_setting}, then the averaged law $\bar\mu_{Nh} \le {(Nh)}^{-1} \int_0^{Nh}\mu_t \, \D t$ satisfies $\norm{\bar\mu_{Nh} - \pi}_{\rm TV}^2 \le \varepsilon$ (for $0 < \varepsilon \le 1$), with \textbf{total gradient complexity} at most
    \begin{align*}
        B\times N
        &\lesssim \begin{dcases}
            \frac{C_{\msf{PI}}^{(1+s)/s} K_0 L^{2/s} d^{3-2s}}{\varepsilon^{(1+s)/s}}\,, & \text{if}~s \ge \frac{1}{2}~\text{with}~B = 1\,, \\[0.25em]
            \frac{C_{\msf{PI}}^3 K_0 L^{6/(1+s)} d^{3-2s}}{\varepsilon^{(5-s)/(1+s)}}\,, & \text{if}~s \le \frac{1}{2}~\text{with}~B\asymp \frac{C_{\msf{PI}} L^{2/(1+s)}}{\varepsilon^{(1-s)/(1+s)}}\,.
        \end{dcases}
    \end{align*}
\end{corollary}

Compared with~\cite[Theorem 7]{chewietal2021lmcpoincare} which has iteration complexity $\widetilde O(C_{\msf{PI}}^{(1+s)/s} L^{2/s} d^{(1+2s)/s}/\varepsilon^{1/s})$, we see that our dependence on every problem parameter is better except for the dependence on the inverse accuracy, for which we obtain a better rate only for $s \le 2-\sqrt 3 \approx 0.27$. In particular, our complexity does not blow up as $s \searrow 0$, so we can set $s= 0$ and get an iteration complexity of $O(C_{\msf{PI}}^2 L^6 d^4/\varepsilon^5)$ for sampling from \textit{Lipschitz potentials satisfying a Poincar\'e inequality}. To the best of our knowledge, this is the first guarantee for this setting.

\subsection{Finite sum setting}
\label{scn:vr}
Finally, we consider the case $V = \frac{1}{n}\sum_{i=1}^n f_i$ is a finite sum involving a large number $n$ of terms, as it is often the case in machine learning. In the big data regime, mini-batch stochastic gradient-based LMC is preferred to vanilla LMC due to reduced per-iteration costs~\cite{brosse2018promises, chatterji2018theory}. However, stochastic gradients obtained by randomly selecting a mini-batch of data do not have a bounded variance in general. Therefore, the generic Theorem~\ref{thm:stochastic_setting} is not applicable to mini-batching in general.

We consider LMC with a variance-reduced stochastic gradient given by the PAGE estimator~\cite{li2021page}. Indeed, in non-convex optimization, the PAGE estimator has been used to reduce the variance in SGD and led to a simple and optimal stochastic non-convex optimization algorithm. We consider a Variance Reduced LMC algorithm:
\begin{align}\label{eq:vr-lmc} \tag{VR-LMC}
    x_{(k+1)h}
    &:= x_{kh} - h \, g_{kh} + \sqrt 2 \, (B_{(k+1)h} - B_{kh})\,,
\end{align}
where $g_{kh}$ is defined by
\begin{equation}
    g_{(k+1)h} := \begin{cases}
        \nabla V(x_{(k+1)h})\,, & \mbox{with probability } p\,, \\
        g_{kh} + \nabla f_i(x_{(k+1) h}) - \nabla f_i(x_{kh})\,, & \mbox{with probability } 1-p\,,
    \end{cases}
\end{equation}
where $i \sim \msf{uniform}([1,\ldots,n])$ and $p \in (0,1]$. Let us describe how $g_{(k+1)h}$ is obtained from $g_{kh}$. Denote $\mathcal F_k = \sigma(g_0, \ldots, g_{kh}, x_0, \ldots,x_{(k+1)h})$. To obtain $g_{(k+1)h}$, one first samples  $B\sim\msf{Bernoulli}(p)$, independent of $\mathcal F_k$. If $B=1$, then $g_{(k+1)h} = \nabla V(x_{(k+1)h})$ and if $B = 0$ then one samples a uniform random variable $i$, independent of $\mathcal F_k$ and independent of $B$, and one sets $g_{(k+1)h}  =g_{kh} + \nabla f_i(x_{(k+1)h}) - \nabla f_i(x_{kh})$.

Assuming that $g_0$ is an unbiased estimate of $\nabla V(x_0)$, then $\E(g_{(k+1)h}) = \E(\nabla V(x_{(k+1)h}))$ by induction. Therefore the PAGE estimator has zero bias. However, its variance is not uniformly bounded in general and Theorem~\ref{thm:stochastic_setting} is not applicable to~\eqref{eq:vr-lmc}. Nevertheless, we obtain a result under the following assumption.

\begin{assumption}\label{as:grad_sto_smooth}
The potential $V$ is a finite sum $V = n^{-1} \sum_{i=1}^n f_i$ and the gradient of $f_i$ is $L$-Lipschitz continuous for every $i \in [n]$.
\end{assumption}

\begin{theorem}\label{thm:minibatch} 
Let ${(\mu_t)}_{t\ge 0}$ denote the law of the interpolation of~\eqref{eq:vr-lmc} and let the potential $V$ satisfy Assumption \ref{as:grad_sto_smooth}. Then, for all $h \in (0, \frac{\sqrt p}{5L})$, we have
 \begin{align*}
         \frac{1}{Nh}\int_{0}^{Nh}  \FI(\mu_t \mmid \pi) \, \D t \leq \frac{2C}{Nh} + \frac{18L^2 dh}{p}\,,
 \end{align*}
 where 
\begin{align*}
    C := \msf{KL}(\mu_{0} \mmid \pi) + \frac{3h}{p}\E[\norm{g_0 - \nabla V(x_0)}^2]\,.
\end{align*}
In particular, if $h = \frac{\sqrt{p C}}{3L \sqrt{Nd}}$ and $N \ge \frac{2C}{d}$, then
\begin{align*}
    \frac{1}{Nh}\int_{0}^{Nh} \FI(\mu_t \mmid \pi)\, \D t \leq 12L\sqrt{\frac{C d}{N p}}\,.
\end{align*}
\end{theorem}

\begin{remark}
We now elaborate on the \textbf{total number of individual gradient evaluations} based on Theorem~\ref{thm:minibatch}. Since $\msf{KL}(\mu_{0} \mmid \pi) = O(d)$, if we assume that a full gradient is computed at the first step, then $N \asymp {L^2 d^2}/(p \varepsilon^2)$
iterations suffice to achieve $\varepsilon$ accuracy: $\frac{1}{Nh}\int_{0}^{Nh} \FI(\mu_t \mmid \pi) \, \D t \le \varepsilon$. At each iteration, the algorithm computes $pn + 1-p = O(pn)$ new gradients in average. 
Therefore, $\varepsilon$ accuracy is achieved after
\begin{equation}
    O(pnN)
    = O\Bigl(\frac{L^2  d^2 n}{\varepsilon^2}\Bigr)
\end{equation}
gradient computations. A direct application of Theorem~\ref{thm:main} would also give a similar order of gradient computations. However, the per-iteration complexity of~\eqref{eq:vr-lmc} is better than that of~\eqref{eq:langevin}, making it easier to apply to problems with large $n$ in practice. For instance, taking $p = 1/(n-1)$, the amortized number of gradient computations per iteration is constant (independent of $n$) equal to 2, which allows for using mini-batch versions of LMC without making the variance boundedness assumption required by Theorem~\ref{thm:stochastic_setting}.
\end{remark}

\section{Conclusion and open questions}\label{scn:conclusion}

In this work, we have initiated the study of non-log-concave sampling by proving that, under the sole assumption that the potential has a Lipschitz gradient, averaged LMC drives the Fisher information w.r.t.\ the target to zero after polynomially many iterations. We have argued that this is the natural sampling analogue of finding approximate first-order stationary points in non-convex optimization.

Although our focus was to work under the minimal assumption of smoothness, surprisingly our analysis yielded new results for sampling from targets satisfying a Poincar\'e inequality, and moreover our results attain state-of-the-art dimension dependence for these settings for LMC.

We believe there are many intriguing directions for future work, and we list a few to conclude.
\begin{enumerate}
    \item (lower bounds) We ask whether one can prove lower bounds on the complexity of outputting a sample whose Fisher information w.r.t.\ the target is $\varepsilon$. Since the setting of this work is fully non-convex, it may be easier to produce lower bound constructions than the strongly log-concave case, in which the theory of lower bounds is nascent~\cite{chewietal2021logconcave1d}.
    \item (improved results and further extensions) Although we have provided results under Hessian smoothness and via variance reduction, our investigation is still preliminary and we believe that these results can be strengthened. Additionally, there are other important extensions to consider; for instance, is there an analogue of second-order stationarity in sampling?
    \item (Poincar\'e case) The iteration complexity we obtained for smooth potentials which satisfy a Poincar\'e inequality (focusing only on dimension and accuracy) is $O(d^2/\varepsilon^2)$, whereas~\cite{chewietal2021lmcpoincare} obtained $\widetilde O(d^3/\varepsilon)$. Is it possible to achieve $\widetilde O(d^2/\varepsilon)$ with a variant of LMC? If so, is averaging necessary?
\end{enumerate}

\medskip{}
\paragraph{Acknowledgments.} We would like to thank Mufan (Bill) Li and Ruoqi Shen for helpful conversations. KB was supported by a seed grant from Center for Data Science and Artificial Intelligence Research,  UC Davis and NSF Grant DMS-2053918. SC was supported by the Department of Defense (DoD) through the National Defense Science \& Engineering Graduate Fellowship (NDSEG) Program. MAE was supported by NSERC Grant [2019-06167], Connaught New Researcher Award,
CIFAR AI Chairs program, and CIFAR AI Catalyst grant. AS was supported by a Simons--Berkeley Research Fellowship. This work was done while several of the authors were visiting the Simons Institute for the Theory of Computing.

\bibliographystyle{alpha}
\bibliography{ref}

\newcommand{\etalchar}[1]{$^{#1}$}
\begin{thebibliography}{CCAY{\etalchar{+}}18}

\bibitem[AGS08]{ambrosio2008gradient}
Luigi Ambrosio, Nicola Gigli, and Giuseppe Savar\'{e}.
\newblock {\em Gradient flows in metric spaces and in the space of probability
  measures}.
\newblock Lectures in Mathematics ETH Z\"{u}rich. Birkh\"{a}user Verlag, Basel,
  second edition, 2008.

\bibitem[BDM18]{brosse2018promises}
Nicolas Brosse, Alain Durmus, and Eric Moulines.
\newblock The promises and pitfalls of stochastic gradient {L}angevin dynamics.
\newblock {\em Advances in Neural Information Processing Systems}, 31, 2018.

\bibitem[BEGK02]{bovier2002metastability}
Anton Bovier, Michael Eckhoff, V{\'e}ronique Gayrard, and Markus Klein.
\newblock Metastability and low lying spectra in reversible {M}arkov chains.
\newblock {\em Communications in Mathematical Physics}, 228(2):219--255, 2002.

\bibitem[BEGK04]{bovier2004metastabilitya}
Anton Bovier, Michael Eckhoff, V{\'e}ronique Gayrard, and Markus Klein.
\newblock Metastability in reversible diffusion processes {I: S}harp
  asymptotics for capacities and exit times.
\newblock {\em Journal of the European Mathematical Society}, 6(4):399--424,
  2004.

\bibitem[Ber18]{bernton2018langevin}
Espen Bernton.
\newblock {L}angevin {M}onte {C}arlo and {JKO} splitting.
\newblock In {\em Conference on Learning Theory (COLT)}, pages 1777--1798,
  2018.

\bibitem[BGK05]{bovier2004metastabilityb}
Anton Bovier, V{\'e}ronique Gayrard, and Markus Klein.
\newblock Metastability in reversible diffusion processes {II: P}recise
  asymptotics for small eigenvalues.
\newblock {\em Journal of the European Mathematical Society}, 7(1):69--99,
  2005.

\bibitem[BGL14]{bakrygentilledoux2014}
Dominique Bakry, Ivan Gentil, and Michel Ledoux.
\newblock {\em Analysis and geometry of {M}arkov diffusion operators}, volume
  348 of {\em Grundlehren der Mathematischen Wissenschaften [Fundamental
  Principles of Mathematical Sciences]}.
\newblock Springer, Cham, 2014.

\bibitem[CCAY{\etalchar{+}}18]{cheng2018sharp}
Xiang Cheng, Niladri~S Chatterji, Yasin Abbasi-Yadkori, Peter~L Bartlett, and
  Michael~I Jordan.
\newblock {Sharp convergence rates for Langevin dynamics in the nonconvex
  setting}.
\newblock {\em arXiv preprint arXiv:1805.01648}, 2018.

\bibitem[CDJB20]{chatterji2020langevin}
Niladri Chatterji, Jelena Diakonikolas, Michael~I Jordan, and Peter Bartlett.
\newblock Langevin {M}onte {C}arlo without smoothness.
\newblock In {\em International Conference on Artificial Intelligence and
  Statistics}, pages 1716--1726. PMLR, 2020.

\bibitem[CEL{\etalchar{+}}21]{chewietal2021lmcpoincare}
Sinho Chewi, Murat~A. Erdogdu, Mufan~B. Li, Ruoqi Shen, and Matthew Zhang.
\newblock Analysis of {L}angevin {M}onte {C}arlo from {P}oincar\'e to
  log-{S}obolev.
\newblock {\em arXiv e-prints}, 2021.

\bibitem[CFM{\etalchar{+}}18]{chatterji2018theory}
Niladri Chatterji, Nicolas Flammarion, Yian Ma, Peter Bartlett, and Michael
  Jordan.
\newblock {On the theory of variance reduction for stochastic gradient Monte
  Carlo}.
\newblock In {\em International Conference on Machine Learning}, pages
  764--773. PMLR, 2018.

\bibitem[CGL{\etalchar{+}}21]{chewietal2021logconcave1d}
Sinho Chewi, Patrik Gerber, Chen Lu, Thibaut~Le Gouic, and Philippe Rigollet.
\newblock The query complexity of sampling from strongly log-concave
  distributions in one dimension.
\newblock {\em arXiv e-prints}, 2021.

\bibitem[DK19]{dalalyan2019user}
Arnak~S Dalalyan and Avetik Karagulyan.
\newblock {User-friendly guarantees for the Langevin Monte Carlo with
  inaccurate gradient}.
\newblock {\em Stochastic Processes and their Applications},
  129(12):5278--5311, 2019.

\bibitem[DM17]{durmus2017nonasymptotic}
Alain Durmus and Eric Moulines.
\newblock Nonasymptotic convergence analysis for the unadjusted {L}angevin
  algorithm.
\newblock {\em The Annals of Applied Probability}, 27(3):1551--1587, 2017.

\bibitem[DMM19]{durmus2019analysis}
Alain Durmus, Szymon Majewski, and B{\l}a{\.z}ej Miasojedow.
\newblock {Analysis of Langevin Monte Carlo via convex optimization}.
\newblock {\em The Journal of Machine Learning Research}, 20(1):2666--2711,
  2019.

\bibitem[DMPS18]{douc2018markov}
Randal Douc, Eric Moulines, Pierre Priouret, and Philippe Soulier.
\newblock {\em Markov chains}.
\newblock Springer, 2018.

\bibitem[DMR20]{devroyemehrabianreddad2020tvgaussians}
Luc Devroye, Abbas Mehrabian, and Tommy Reddad.
\newblock The total variation distance between high-dimensional {G}aussians.
\newblock {\em arXiv e-prints}, 2020.

\bibitem[EH21]{erdogdu2021convergence}
Murat~A Erdogdu and Rasa Hosseinzadeh.
\newblock {On the convergence of Langevin Monte Carlo: the interplay between
  tail growth and smoothness}.
\newblock In {\em Conference on Learning Theory}, pages 1776--1822. PMLR, 2021.

\bibitem[EMS18]{erdogdu2018global}
Murat~A Erdogdu, Lester Mackey, and Ohad Shamir.
\newblock Global non-convex optimization with discretized diffusions.
\newblock {\em Advances in Neural Information Processing Systems}, 31, 2018.

\bibitem[GLWY09]{guillinetal2009transportinfo}
Arnaud Guillin, Christian L\'{e}onard, Liming Wu, and Nian Yao.
\newblock Transportation-information inequalities for {M}arkov processes.
\newblock {\em Probab. Theory Related Fields}, 144(3-4):669--695, 2009.

\bibitem[HBE22]{he2022heavy}
Ye~He, Krishnakumar Balasubramanian, and Murat~A. Erdogdu.
\newblock Heavy-tailed sampling via transformed unadjusted {L}angevin
  algorithm.
\newblock {\em arXiv preprint arXiv:2201.08349}, 2022.

\bibitem[HW17]{holmes2017assigning}
Chris~C. Holmes and Stephen~G. Walker.
\newblock Assigning a value to a power likelihood in a general {B}ayesian
  model.
\newblock {\em Biometrika}, 104(2):497--503, 2017.

\bibitem[JKO98]{jko1998}
Richard Jordan, David Kinderlehrer, and Felix Otto.
\newblock The variational formulation of the {F}okker-{P}lanck equation.
\newblock {\em SIAM J. Math. Anal.}, 29(1):1--17, 1998.

\bibitem[JNG{\etalchar{+}}21]{jin2021nonconvex}
Chi Jin, Praneeth Netrapalli, Rong Ge, Sham~M Kakade, and Michael~I Jordan.
\newblock On nonconvex optimization for machine learning: Gradients,
  stochasticity, and saddle points.
\newblock {\em Journal of the ACM (JACM)}, 68(2):1--29, 2021.

\bibitem[KNS16]{kariminutinischmidt2016pl}
Hamed Karimi, Julie Nutini, and Mark Schmidt.
\newblock Linear convergence of gradient and proximal-gradient methods under
  the {P}olyak-{L}ojasiewicz condition.
\newblock In {\em European Conference on Machine Learning and Knowledge
  Discovery in Databases}, page 795–811, 2016.

\bibitem[KY03]{kus-yin-livre03}
Harold.~J. Kushner and George~G. Yin.
\newblock {\em Stochastic approximation and recursive algorithms and
  applications}, volume~35.
\newblock Springer-Verlag, New York, second edition, 2003.

\bibitem[LBZR21]{li2021page}
Zhize Li, Hongyan Bao, Xiangliang Zhang, and Peter Richt{\'a}rik.
\newblock {PAGE}: A simple and optimal probabilistic gradient estimator for
  nonconvex optimization.
\newblock In {\em International Conference on Machine Learning}, pages
  6286--6295. PMLR, 2021.

\bibitem[Loj63]{lojasiewicz1963topological}
Stanislaw Lojasiewicz.
\newblock A topological property of real analytic subsets.
\newblock {\em Coll. du CNRS, Les {\'e}quations aux d{\'e}riv{\'e}es
  partielles}, 117(87-89):2, 1963.

\bibitem[LP02]{lamberton2002recursive}
Damien Lamberton and Gilles Pages.
\newblock Recursive computation of the invariant distribution of a diffusion.
\newblock {\em Bernoulli}, pages 367--405, 2002.

\bibitem[LWME19]{li2019stochastic}
Xuechen Li, Yi~Wu, Lester Mackey, and Murat~A. Erdogdu.
\newblock Stochastic {R}unge-{K}utta accelerates {L}angevin {M}onte {C}arlo and
  beyond.
\newblock {\em Advances in Neural Information Processing Systems}, 32, 2019.

\bibitem[MCC{\etalchar{+}}21]{ma2021there}
Yi-An Ma, Niladri~S Chatterji, Xiang Cheng, Nicolas Flammarion, Peter~L
  Bartlett, and Michael~I Jordan.
\newblock Is there an analog of {N}esterov acceleration for gradient-based
  {MCMC}?
\newblock {\em Bernoulli}, 27(3):1942--1992, 2021.

\bibitem[MFWB19]{mou2019improved}
Wenlong Mou, Nicolas Flammarion, Martin~J Wainwright, and Peter~L Bartlett.
\newblock {Improved bounds for discretization of Langevin diffusions:
  Near-optimal rates without convexity}.
\newblock {\em arXiv preprint arXiv:1907.11331}, 2019.

\bibitem[MMS20]{majka2020nonasymptotic}
Mateusz~B Majka, Aleksandar Mijatovi{\'c}, and {\L}ukasz Szpruch.
\newblock Nonasymptotic bounds for sampling algorithms without log-concavity.
\newblock {\em The Annals of Applied Probability}, 30(4):1534--1581, 2020.

\bibitem[N{\etalchar{+}}18]{nesterov2018lectures}
Yurii Nesterov et~al.
\newblock {\em Lectures on convex optimization}, volume 137.
\newblock Springer, 2018.

\bibitem[NS17]{nesterov2017random}
Yurii Nesterov and Vladimir Spokoiny.
\newblock Random gradient-free minimization of convex functions.
\newblock {\em Foundations of Computational Mathematics}, 17(2):527--566, 2017.

\bibitem[Pol63]{polyak1963gradient}
Boris Polyak.
\newblock Gradient methods for minimizing functionals.
\newblock {\em Zhurnal Vychislitel'noi Matematiki i Matematicheskoi Fiziki},
  3(4):643--653, 1963.

\bibitem[PP12]{pages2012ergodic}
Gilles Pag{\`e}s and Fabien Panloup.
\newblock Ergodic approximation of the distribution of a stationary diffusion:
  rate of convergence.
\newblock {\em The Annals of Applied Probability}, 22(3):1059--1100, 2012.

\bibitem[RRT17]{raginskyrakhlintelgarsky2017sgld}
Maxim Raginsky, Alexander Rakhlin, and Matus Telgarsky.
\newblock Non-convex learning via stochastic gradient {L}angevin dynamics: a
  nonasymptotic analysis.
\newblock In {\em Proceedings of the Conference on Learning Theory}, volume~65
  of {\em PMLR}, pages 1674--1703, 2017.

\bibitem[San15]{santambrogio2015ot}
Filippo Santambrogio.
\newblock {\em Optimal transport for applied mathematicians}, volume~87 of {\em
  Progress in Nonlinear Differential Equations and their Applications}.
\newblock Birkh\"{a}user/Springer, Cham, 2015.
\newblock Calculus of variations, PDEs, and modeling.

\bibitem[SJDT19]{shao2019bayesian}
Stephane Shao, Pierre~E Jacob, Jie Ding, and Vahid Tarokh.
\newblock {Bayesian model comparison with the Hyv{\"a}rinen score: Computation
  and consistency}.
\newblock {\em Journal of the American Statistical Association}, 2019.

\bibitem[Vil09]{villani2009ot}
C\'{e}dric Villani.
\newblock {\em Optimal transport}, volume 338 of {\em Grundlehren der
  Mathematischen Wissenschaften [Fundamental Principles of Mathematical
  Sciences]}.
\newblock Springer-Verlag, Berlin, 2009.
\newblock Old and new.

\bibitem[VW19]{vempala2019ulaisoperimetry}
Santosh Vempala and Andre Wibisono.
\newblock Rapid convergence of the unadjusted {L}angevin algorithm:
  {I}soperimetry suffices.
\newblock In {\em Advances in Neural Information Processing Systems 32}, pages
  8094--8106. 2019.

\bibitem[Wal16]{walker2016bayesian}
Stephen~G Walker.
\newblock Bayesian information in an experiment and the {F}isher information
  distance.
\newblock {\em Statistics \& Probability Letters}, 112:5--9, 2016.

\bibitem[Wib18]{wibisono2018sampling}
Andre Wibisono.
\newblock Sampling as optimization in the space of measures: The {L}angevin
  dynamics as a composite optimization problem.
\newblock In {\em Conference on Learning Theory}, pages 2093--3027. PMLR, 2018.

\bibitem[Wu00]{wu2000largedeviations}
Liming Wu.
\newblock Uniformly integrable operators and large deviations for {M}arkov
  processes.
\newblock {\em J. Funct. Anal.}, 172(2):301--376, 2000.

\end{thebibliography}

\pagebreak 
\appendix

\section{Proof for the illustrative example}

\begin{proof}[Proof of Proposition~\ref{prop:example}]
    The total variation distance is
    \begin{align*}
        \norm{\mu-\pi}_{\rm TV}
        &= \frac{1}{2} \int \abs{\mu-\pi}
        = \frac{1}{8} \int \abs{\pi_+ - \pi_-}
        = \frac{1}{4} \, \norm{\pi_+ - \pi_-}_{\rm TV}\,.
    \end{align*}
    Since $\pi_- = \normal(-m, 1)$ and $\pi_+ = \normal(m, 1)$, the lower bound on $\norm{\mu-\pi}_{\rm TV}$ follows from~\cite[Theorem 1.3]{devroyemehrabianreddad2020tvgaussians}.
    
    Next, we have
    \begin{align*}
        \nabla \ln \frac{\mu}{\pi}
        &= \frac{1}{\mu} \, \bigl(\frac{3}{4} \, \nabla \pi_- + \frac{1}{4} \,\nabla\pi_+\bigr) - \frac{1}{\pi} \, \bigl( \frac{1}{2} \, \nabla\pi_- + \frac{1}{2} \, \nabla\pi_+\bigr) \\
        &= \frac{1}{\mu\pi} \, \bigl[ \pi \, \bigl(\frac{3}{4} \, \nabla\pi_- + \frac{1}{4} \,\nabla\pi_+\bigr) + \mu \, \bigl( \frac{1}{2} \, \nabla\pi_- + \frac{1}{2} \, \nabla\pi_+\bigr)\bigr]\,.
    \end{align*}
    Writing $s_{\mp} := \nabla \ln \pi_{\mp}$, some algebra reveals that
    \begin{align*}
        \nabla \ln \frac{\mu}{\pi}
        &= \frac{1}{4\mu\pi} \, (\pi_+ \, \nabla \pi_- - \pi_- \, \nabla \pi_+)
        = \frac{\pi_- \pi_+}{4\mu\pi} \, (s_- - s_+)
        = -\frac{\pi_- \pi_+}{2\mu\pi} \, m\,.
    \end{align*}
    Therefore,
    \begin{align*}
        \FI(\mu\mmid \pi)
        &= \frac{m^2}{4} \int \frac{\pi_-^2 \pi_+^2}{\mu^2 \pi^2} \, \D \mu
        = \frac{m^2}{4} \int \frac{\pi_-^2 \pi_+^2}{\mu \pi^2}
        = \frac{m^2}{4} \int \frac{\pi_-^2 \pi_+^2}{(\frac{3}{4} \, \pi_- + \frac{1}{4} \, \pi_+)\, {(\frac{1}{2} \,\pi_- + \frac{1}{2} \, \pi_+)}^2} \\
        &\le 4m^2 \int \frac{\pi_-^2 \pi_+^2}{{(\pi_- + \pi_+)}^3}
        \le 4m^2 \, \Bigl[ \int_{\R_-} \frac{\pi_+^2}{\pi_-} + \int_{\R_+} \frac{\pi_-^2}{\pi_+} \Bigr]\,.
    \end{align*}
    Writing $Z := {(2\uppi)}^{d/2}$ for the normalizing constant,
    \begin{align*}
        \int_{\R_+} \frac{\pi_-^2}{\pi_+}
        &= \frac{1}{Z} \int_0^\infty \exp\bigl(-\abs{x+m}^2 + \frac{1}{2} \, \abs{x-m}^2\bigr) \, \D x
        = \frac{\exp(4m^2)}{Z} \int_0^\infty \exp\bigl(-\frac{1}{2} \, \abs{x+3m}^2\bigr) \, \D x \\
        &= \exp(4m^2) \, \Pr\{\xi \ge 3m\}
    \end{align*}
    where $\xi$ is a standard Gaussian random variable. Using a standard Gaussian tail bound,
    \begin{align*}
        \Pr\{\xi \ge 3m\}
        &\le \frac{1}{2} \exp\bigl( -\frac{9m^2}{2}\bigr)\,.
    \end{align*}
    A symmetric argument holds for the other integral, and hence
    \begin{align*}
        \FI(\mu\mmid\pi)
        &\le 4m^2 \exp\bigl(-\frac{m^2}{2}\bigr)
    \end{align*}
    which completes the proof.
\end{proof}

\section{Proof of the main theorem}

Our proof follows the interpolation argument of~\cite{vempala2019ulaisoperimetry} which proceeds by obtaining a differential inequality for the KL divergence along an interpolation of the algorithm. With an eye towards extensions of the main result, we prove a more general version of the inequality.

\begin{lemma}\label{lem:fokker_planck}
    Consider the stochastic process defined by
    \begin{align*}
        x_t
        &:= x_0 - hg_0 + \sqrt 2 \, B_t\,, \qquad\text{for}~t\ge 0\,,
    \end{align*}
    where ${(B_t)}_{t\ge 0}$ is a standard Brownian motion in $\R^d$ which is independent of $(x_0, g_0)$. Then, writing $\mu_t$ for the law of $x_t$,
    \begin{align*}
        \partial_t \KL(\mu_t \mmid \pi)
        &\le - \frac{3}{4} \FI(\mu_t \mmid \pi) + \E\bigl[\norm{\nabla V(x_t) - \E[g_0\mid x_t]}^2\bigr] \\
        &\le - \frac{3}{4} \FI(\mu_t \mmid \pi) + \E[\norm{\nabla V(x_t) - g_0}^2]\,.
    \end{align*}
\end{lemma}
\begin{proof}
    Let $\mc F_0$ denote the $\sigma$-algebra generated by $(x_0,g_0)$, and let $\mu_{t \mid \mc F_0}$ denote the conditional law of $x_t$ given $\mc F_0$. Then, $t\mapsto \mu_{t\mid \mc F_0}$ evolves according to the Fokker-Planck equation
    \begin{align*}
        \partial_t \mu_{t\mid \mc F_0}(x)
        &= \Delta \mu_{t\mid \mc F_0}(x) + \divergence_x\bigl(\mu_{t\mid \mc F_0}(x)\, g_0\bigr)\,.
    \end{align*}
    If $\Pr_0$ denotes the restriction of the probability measure $\Pr$ on the underlying probability space, then taking the expectation w.r.t.\ $\Pr_0$ yields
    \begin{align*}
        \partial_t \mu_t(x)
        &= \Delta \mu_t(x) + \divergence_x \E[\mu_{t\mid \mc F_0}(x) \, g_0]\,.
    \end{align*}
    The second term is
    \begin{align*}
        \E[\mu_{t\mid \mc F_0}(x) g_0]
        &= \int \mu_{t\mid \mc F_0}(x\mid \omega) \, g_0(\omega) \, \Pr_0(\D \omega)
        = \mu_t(x) \int g_0(\omega) \, \mu_{\mc F_0 \mid t}(\D \omega \mid x) \\
        &= \mu_t(x) \E[g_0 \mid x_t = x]\,.
    \end{align*}
    From this, the time derivative of the KL divergence is
    \begin{align*}
        \partial_t \KL(\mu_t \mmid \pi)
        &= \int \bigl(\ln \frac{\mu_t}{\pi}\bigr) \divergence\bigl(\mu_t \, (\nabla \ln \mu_t + \E[g_0 \mid x_t = \cdot])\bigr) \\
        &= -\int \bigl\langle \nabla \ln \frac{\mu_t}{\pi}, \nabla \ln \mu_t + \E[g_0 \mid x_t = \cdot]\bigr\rangle \, \D \mu_t \\
        &= -\FI(\mu_t \mmid \pi) + \int \bigl\langle \nabla \ln \frac{\mu_t}{\pi}, \nabla V -  \E[g_0 \mid x_t = \cdot]\bigr\rangle \, \D \mu_t\,.
    \end{align*}
    Applying Young's inequality,
    \begin{align*}
        \int \bigl\langle \nabla \ln \frac{\mu_t}{\pi}, \nabla V -  \E[g_0 \mid x_t = \cdot]\bigr\rangle \, \D \mu_t
        &\le \frac{1}{4} \FI(\mu_t \mmid \pi) + \E\bigl[\norm{\nabla V(x_t) - \E[g_0 \mid x_t]}^2\bigr]
    \end{align*}
    which completes the proof.
\end{proof}

We also use the following lemma, which is taken from~\cite{chewietal2021lmcpoincare}. For the reader's convenience, the proof is reproduced here.

\begin{lemma}[{\cite[Lemma 16]{chewietal2021lmcpoincare}}]\label{lem:fisher_info_arg}
    Assume that $\nabla V$ is $L$-Lipschitz.
    For any probability measure $\mu$, it holds that
    \begin{align*}
        \E_\mu[\norm{\nabla V}^2]
        &\le \FI(\mu\mmid\pi) + 2dL\,.
    \end{align*}
\end{lemma}
\begin{proof}
    Let $\ms L$ denote the infinitesimal generator of the Langevin diffusion~\eqref{eq:langevin}, i.e.
    \begin{align*}
        \ms L f &:= \langle \nabla V, \nabla f \rangle - \Delta f\,.
    \end{align*}
    Observe that $\ms L V = \norm{\nabla V}^2 - \Delta V$.
    Applying integration by parts,
    \begin{align*}
        \E_\mu[\norm{\nabla V}^2]
        &= \E_\mu \ms L V + \E_\mu \Delta V
        \le \int \ms L V \, \frac{\D\mu}{\D\pi} \, \D \pi + dL
        = \int \bigl\langle \nabla V, \nabla \frac{\D\mu}{\D\pi} \bigr\rangle  \, \D \pi + dL \\
        &= 2 \int \bigl\langle \sqrt{\frac{\D\mu}{\D\pi}} \, \nabla V, \nabla \sqrt{\frac{\D\mu}{\D\pi}} \bigr\rangle \, \D \pi + dL
        \le \frac{1}{2} \E_\mu[\norm{\nabla V}^2] + 2 \E_\pi\bigl[\bigl\lVert \nabla \sqrt{\frac{\D\mu}{\D\pi}} \bigr\rVert^2 \bigr] + dL\,.
    \end{align*}
    Rearrange this inequality to obtain the desired result.
\end{proof}

We now prove our main result.
\medskip{}

\begin{proof}[Proof of Theorem~\ref{thm:main}]
    Let ${(x_t)}_{t\ge 0}$ denote the interpolation of LMC (defined in~\eqref{eq:interpolated_lmc}).
    For $t \in [kh, (k+1)h]$, Lemma~\ref{lem:fokker_planck} yields
    \begin{align*}
        \partial_t \KL(\mu_t \mmid \pi)
        &\le - \frac{3}{4} \FI(\mu_t\mmid\pi) + \E[\norm{\nabla V(x_t) - \nabla V(x_{kh})}^2]
    \end{align*}
    and the error term is
    \begin{align*}
        \E[\norm{\nabla V(x_t) - \nabla V(x_{kh})}^2]
        &\le L^2 \E[\norm{x_t - x_{kh}}^2] \\
        &\le 2L^2 \, {(t-kh)}^2 \E[\norm{\nabla V(x_{kh})}^2] + 4L^2 \E[\norm{B_t - B_{kh}}^2]\,.
    \end{align*}
    Next, since $\nabla V$ is Lipschitz,
    \begin{align*}
        \norm{\nabla V(x_{kh})}
        &\le \norm{\nabla V(x_t)} + L \, \norm{x_t - x_{kh}} \\
        &\le \norm{\nabla V(x_t)} + L h \, \norm{\nabla V(x_{kh})} + \sqrt 2 L \, \norm{B_t - B_{kh}}\,.
    \end{align*}
    and for $h\le 1/(3L)$ we can rearrange this to yield
    \begin{align*}
        \norm{\nabla V(x_{kh})}
        &\le \frac{3}{2} \, \norm{\nabla V(x_t)} + \frac{3L}{\sqrt 2} \, \norm{B_t - B_{kh}}\,.
    \end{align*}
    Plugging this in,
    \begin{align}\label{eq:disc_error_bd}
        \norm{\nabla V(x_t) - \nabla V(x_{kh})}^2
        &\le 9L^2 \, {(t-kh)}^2 \, \norm{\nabla V(x_t)}^2 + 6L^2 \, \norm{B_t - B_{kh}}^2\,.
    \end{align}
    For the expectation of the first term, we can use Lemma~\ref{lem:fisher_info_arg} to bound
    \begin{align*}
        \E_{\mu_t}[\norm{\nabla V}^2]
        &\le \FI(\mu_t \mmid \pi) + 2Ld\,.
    \end{align*}
    Hence, for $h \le 1/(6L)$,
    \begin{align}
    \label{eq:inf}
        \partial_t \KL(\mu_t \mmid \pi)
        &\le - \bigl( \frac{3}{4} - 9L^2 h^2\bigr) \FI(\mu_t\mmid\pi)
        + 18L^3 d \, {(t-kh)}^2 + 6L^2 d \, (t-kh) \nonumber \\
        &\le - \frac{1}{2} \FI(\mu_t\mmid\pi) + 18L^3 d \, {(t-kh)}^2 + 6L^2 d \, (t-kh)\,.
    \end{align}
    Integrating, we obtain
    \begin{align}
    \label{eq:onesteprec}
        \KL(\mu_{(k+1)h} \mmid \pi) - \msf{KL}(\mu_{kh} \mmid \pi)
        &\le - \frac{1}{2} \int_{kh}^{(k+1)h} \FI(\mu_t\mmid\pi) \, \D t + 6L^3 dh^3 + 3L^2 dh^2 \nonumber\\
        &\le - \frac{1}{2} \int_{kh}^{(k+1)h} \FI(\mu_t\mmid\pi) \, \D t + 4L^2 dh^2\,.
    \end{align}
    Now by summing, we have
    \begin{align*}
        \frac{1}{Nh} \int_0^{Nh} \FI(\mu_t\mmid\pi) \, \D t
        &\le \frac{2 \, \msf{KL}(\mu_0 \mmid \pi)}{Nh} + 8L^2 dh\,.
    \end{align*}
    This concludes the proof.
\end{proof}

\section{Proofs for the extensions and applications}

\subsection{Asymptotic convergence of averaged LMC}

\begin{proof}[Proof of Theorem~\ref{thm:asymptotic}]
The one-step recursion~\eqref{eq:onesteprec} in the proof of Theorem~\ref{thm:main} yields
\begin{align*}
        \KL(\mu_{\tau_n} \mmid \pi) - \KL(\mu_{\tau_{n-1}} \mmid \pi)
        &\le - \frac{1}{2} \int_{\tau_{n-1}}^{\tau_n} \FI(\mu_t\mmid\pi) \, \D t + 4L^2 dh_n^2\,.
\end{align*}
Iterating the above bound, we obtain
\begin{align*}
    \msf{KL}(\mu_{\tau_n} \mmid \pi) \leq \msf{KL}(\mu_{0} \mmid \pi) -\frac{1}{2} \int_{0}^{\tau_n} \FI(\mu_t\mmid\pi) \, \D t + 4L^2d \sum_{k=1}^n h_k^2\,.
\end{align*}
Rearranging the terms, dividing by $\tau_n$, and using the convexity of the Fisher information,
\begin{align}\label{eq:asymptotic_arg}
    \FI(\bar\mu_{\tau_n}\mmid \pi)
    &\le \frac{1}{\tau_n} \int_{0}^{\tau_n} \FI(\mu_t\mmid\pi) \, \D t \leq \frac{2\KL(\mu_{0} \mmid \pi)}{\tau_n}  +   \frac{8L^2 d}{\tau_n}\,S\,,
\end{align}
where $S \coloneqq \sum_{k=1}^{\infty} h_k^2 < \infty$.
On the other hand, if $t \in [\tau_n,\tau_{n+1}]$, integrating~\eqref{eq:inf} between $\tau_n$ and $t$ shows that
\begin{align*}
    \KL(\mu_t\mmid \pi)
    &\le \KL(\mu_{\tau_n}\mmid\pi) + 4L^2 d \, {(t-\tau_n)}^2
    \le \KL(\mu_0\mmid \pi) + 8L^2 d S
    < \infty\,,
\end{align*}
so that $\{\KL(\mu_t\mmid\pi) \mid t\ge 0\}$ is bounded. By convexity of the KL divergence, it also implies that $\{\KL(\bar\mu_{\tau_n} \mmid \pi) \mid n\in \N\}$ is uniformly bounded. Recalling that the sublevel sets of $\KL(\cdot\mmid\pi)$ are weakly compact we obtain that ${(\bar\mu_{\tau_n})}_{n\in\N}$ is tight. To show that $\bar\mu_{\tau_n}\to\pi$ weakly, it suffices to show that every cluster point of ${(\bar\mu_{\tau_n})}_{n\in\N}$ is equal to $\pi$. Consider a subsequence of ${(\bar\mu_{\tau_n})}_{n\in\N}$ converging to some cluster point $\bar{\mu}$.


Taking $n\to\infty$ in~\eqref{eq:asymptotic_arg} and noting that $\tau_n\to\infty$ by our assumptions, we have $\FI(\bar\mu_{\tau_{n}} \mmid \pi) \to 0$, therefore this is still true along the subsequence. Using the weak lower semicontinuity of the Fisher information along the subsequence, $\FI(\bar\mu\mmid\pi) = 0$.
This means that for $f := \frac{\D\bar\mu}{\D\pi}$, we have $\sqrt f \in \dom \ms E$ and $\ms E(\sqrt f) = 0$.
Since $\nabla V$ is Lipschitz, then $\pi$ has a continuous and strictly positive density on $\R^d$, so $\ms E(\sqrt f) = 0$ implies that $f$ is a constant $\pi$-a.e., and hence $\bar\mu = \pi$.
\end{proof}

\subsection{Hessian smoothness}

We first control the moments of LMC under Assumption~\ref{as:hess_growth}.

\begin{proposition}\label{prop:moment-bounds}
Assume that the growth conditions in \eqref{eq:hess_growth} are satisfied for $\gamma > 0$, $0 < \xi \leq \gamma/2$, and $h \leq \frac{a}{4\constL^2} \wedge 1$. 
Then, for the LMC iterates ${(x_{kh})}_{k\in\N}$, we have
\begin{align}
    \E[\norm{x_{kh}}^2] &\leq \E[\norm{x_0}^2] + 3 \, (a+b+d) \,kh\,,\\
    \E[\norm{x_{kh}}^4] &\leq \E[\norm{x_0}^4] + 6\, \Bigl(\frac{3 \, (\constA + \constB + d)}{1 \wedge a}\Bigr)^{\frac{2+\gamma}{\gamma} \vee 2}\,kh\,.
\end{align}
\end{proposition}
\begin{proof}
As before, we denote the interpolation diffusion with ${\{x_t\}}_{t\ge 0}$ and the corresponding filtration with ${\{\mathcal F_t\}}_{t\ge 0}$. Using It\^o's formula conditioned on $\mathcal F_{kh}$, we obtain
\begin{align*}
\partial_t\E[\norm{x_t}^2 \mid \mathcal F_{kh}]
    &= -2\E[\langle x_t,\nabla V(x_{kh}) \rangle \mid \mathcal F_{kh}] +2d\\
    &= - 2\,\langle x_{kh}, \nabla V(x_{kh}) \rangle + 2\,(t-kh)\, \norm{\nabla V(x_{kh})}^2 +2d\\
    &\leq -2\constA\,(1+\|x_{kh}\|^\constG)+ 2\,(\constA+\constB+d) + 4\, (t-kh) \,\constL^2 \, (1+\norm{x_{kh}}^{2\constZ})\\
    &\leq 3\constA+2\constB+2d
\end{align*}
where we used $h\leq \constA/(4\constL^2)$. Integrating this from $kh$ to $(k+1)h$ and iterating yields the second moment bound in \eqref{prop:moment-bounds}.

Similarly for the fourth moment, we write
\begin{align*}
    \partial_t \E[\norm{x_t}^4 \mid \mathcal F_{kh}]
     &= -4\E[\norm{x_t}^2 \,\langle \nabla V(x_{kh}), x_t\rangle \mid \mathcal F_{kh}]  + (4d+2)\E[\norm{x_t}^2 \mid \mathcal F_{kh}]\\
    & =4\E[\norm{x_t}^2\mid  \mathcal F_{kh}]\, \{ -\langle \nabla V(x_{kh}), x_{kh}\rangle 
+(t-kh)\,\norm{\nabla V(x_{kh})}^2 +d+1/2\}\\
&\qquad\ \ \ \, -16 \,(t-kh)\, \langle x_{kh} - (t-kh)\,\nabla V(x_{kh}), \nabla V(x_{kh}) \rangle
\end{align*}
where in the last step, we use Gaussian integration by parts 
\begin{align*}
    &-4\E[\norm{x_t}^2 \, \langle \nabla V(x_{kh}), \sqrt 2 \, (B_t - B_{kh})\rangle \mid \mc F_{kh}] \\
    &\qquad = -16\, (t-kh) \,\langle x_{kh} - (t-kh)\,\nabla V(x_{kh}), \nabla V(x_{kh}) \rangle\,.
\end{align*}
Therefore, we can use the growth condition and write
\begin{align*}
    \partial_t \E[\norm{x_t}^4 \mid \mc F_{kh}]
    &\le 4 \, \{\E[\norm{x_t}^2 \mid \mc F_{kh}] + 4 \, (t-kh)\} \\
    &\qquad{}\times \{ -\langle \nabla V(x_{kh}), x_{kh}\rangle 
+(t-kh)\,\norm{\nabla V(x_{kh})}^2 +d+1/2\} \\
    &\le 4 \, \{\E[\norm{x_t}^2 \mid \mc F_{kh}] + 4 \, (t-kh)\} \\
    &\qquad{}\times \{-\constA \, \norm{x_{kh}}^\constG + \constB + 2 \, (t-kh) \, \constL^2 \, (1+\norm{x_{kh}}^{2\constZ}) + d + 1/2\}\,.
\end{align*}
Next, recalling that $h \le \min\{1, a/(4m^2)\}$ and using Assumption~\ref{as:hess_growth},
\begin{align*}
    \partial_t \E[\norm{x_t}^4 \mid \mc F_{kh}]
    &\le 4 \, \{\E[\norm{x_t}^2 \mid \mc F_{kh}] + 4 \, (t-kh)\}
    \times \bigl\{-\frac{\constA}{2} \, \norm{x_{kh}}^\constG + \frac{\constA}{2} + \constB + d + \frac{1}{2}\bigr\}\,.
\end{align*}
Define $C := \constA + \constB + d$. We split into two cases.
If $\norm{x_{kh}} \ge {((\constA + 2\constB + 2d + 1)/\constA)}^{1/\gamma}$, then the time derivative is \emph{negative}.
Otherwise, if $\norm{x_{kh}} \le {((\constA + 2\constB + 2d + 1)/\constA)}^{1/\gamma} \le {(3C/\constA)}^{1/\gamma}$, then recalling our second moment bound,
\begin{align*}
    \E[\norm{x_t}^2 \mid \mc F_{kh}] + 4 \, (t-kh)
    &\le \norm{x_{kh}}^2 + (t-kh) \, (3\constA + 2\constB + 2d + 4)
    \le \norm{x_{kh}}^2 + 6C
\end{align*}
and therefore
\begin{align*}
    \partial_t \E[\norm{x_t}^4 \mid \mc F_{kh}]
    &\le 3 \, \Bigl( \bigl( \frac{3C}{\constA} \bigr)^{2/\gamma} + 6C \Bigr) \, C
    \le 6 \, \bigl(\frac{3C}{1\wedge a} \bigr)^{\frac{2+\gamma}{\gamma} \vee 2}\,.
\end{align*}

This concludes the proof.
\end{proof}

\begin{proof}[Proof of Theorem~\ref{thm:hess_smooth}]
Under Hessian smoothness, we can achieve tighter control on the discretization error via the fourth moment. To do this, we introduce the following lemma, which is derived from an intermediate result in the work of \cite{mou2019improved}.
\begin{lemma}
    \label{lem:mou_disc_bound}
Under Assumption \ref{as:hess_smooth}, the following bound holds for the discretization error.
\begin{align*}
    &\E\bigl[\norm{\nabla V(x_t) - \E[\nabla V(x_{kh})\mid x_t]}^2\bigr] \\
    &\qquad \leq 4L^2\, (t-kh)^2 \E[\norm{\nabla \ln \mu_{kh}(x_{kh})}^2]  + 12 L^4 \,(t-kh)^3 \,d   + 4 L^2 h^2 \E[\norm{\nabla V(x_{kh})}^2] \\
    &\qquad \qquad + 4 \,(t-kh)^4\, M^2 \E[\norm{\nabla V(x_{kh})}^4] + 48\, (t-kh)^2\, M^2 d^2\,.
\end{align*}
\end{lemma}

\begin{proof}
    This result follows from \cite{mou2019improved}, by combining the proof of their Lemma 3 (before substitution of $\norm{\nabla V(x_{kh})}^4$), their Lemma 4, and the result for the term $I_2$ in their Lemma 5 with the bound on $I_3$ in the proof of their Lemma 5 before substituting for $\norm{\nabla V(x_{kh})}^2$. 
\end{proof}
We invoke the following Lemma, also from \cite{mou2019improved}.
\begin{lemma}[{\cite[Lemma 7]{mou2019improved}}]
    \label{lem:mou_time_diff}
     For $h \leq \frac{1}{2L}$ and all $t \in [kh, (k+1)h]$,
    \begin{align*}
        \E[\norm{\nabla \ln \mu_{kh}(x_{kh})}^2] & \leq 8\E[\norm{\nabla \ln \mu_t(x_t)}^2] + 32 M^2 d^2 h^2 \,.
    \end{align*}
\end{lemma}
Consequently, we first analyze the first term in Lemma \ref{lem:mou_disc_bound} for $t \in [kh, (k+1)h]$:
\begin{align*}
    \E[\norm{\nabla \ln \mu_{kh}(x_{kh})}^2]
    &\lesssim \E[\norm{\nabla \ln \mu_t(x_t)}^2] + M^2 d^2 h^2 \\
    &= \E\bigl[\bigl\lVert \nabla \ln \frac{\mu_t}{\pi}(x_t) + \nabla V(x_t)\bigr\rVert^2\bigr] + M^2 d^2 h^2 \\
    &\lesssim \E\bigl[\bigl\lVert \nabla \ln \frac{\mu_t}{\pi}(x_t)\bigr\rVert^2\bigr] + \E[\norm{\nabla V(x_t)}^2] + M^2 d^2 h^2 \\
    &\lesssim \FI(\mu_t \mmid \pi) + L d + M^2 d^2 h^2\,,
\end{align*}
where we applied Lemma~\ref{lem:fisher_info_arg}.
Similarly, we bound the term
\begin{align*}
    \E[\norm{\nabla V(x_{kh})}^2]
    &\lesssim \E[\norm{\nabla V(x_t)}^2] + \E[\norm{\nabla V(x_{kh}) - \nabla V(x_t)}^2] \\
    &\lesssim (1 + L^2 \, {(t-kh)}^2) \,\E[\norm{\nabla V(x_t)}^2] + L^2\, (t-kh)\,d \\
    &\leq \FI(\mu_t \mmid \pi) + Ld\,,
\end{align*}
where we used~\eqref{eq:disc_error_bd}, $h\lesssim 1/L$, and Lemma~\ref{lem:fisher_info_arg}.

The primary term of concern is the expected fourth power of the gradient, $\E[\norm{\nabla V(x_{kh})}^4]$. For large orders of growth $\xi > 1/2$, we can directly use the fourth moment bound found in Proposition \ref{prop:moment-bounds}, which has a worst case order of $d^3$. However, when $\xi \leq 1/2$, the term $\E[\norm{\nabla V(x_{kh})}^4]$ will only grow as the second moment $\E[\norm{x_{kh}}^2]$, and consequently the order of this term is $d$. In both cases, this term is no longer dominant.


\textbf{Case $\xi > 1/2$:} Using our growth assumption, we get using Assumption \ref{as:hess_growth} for $\xi > 1/2$
\begin{align*}
    \E[\norm{\nabla V(x_{kh})}^4]
    &\lesssim \constL^4 \,(1+ \norm{x_{kh}}^4) \\
    &\lesssim \constL^4\,\biggl(1+\E[\norm{x_0}^4] + \Bigl( \frac{3 \, (\constA+\constB+d)}{1 \wedge a} \Bigr)^{\frac{2+\gamma}{\gamma} \vee 2} \, kh\biggr) \\
    &\lesssim \constL^4\,\bigl(1+\E[\norm{x_0}^4] + (\constB +d)^3 \, kh\bigr)\,,
\end{align*}
where the last line follows as $\xi > 1/2$ implies $\gamma > 1$.

\textbf{Case $\xi \leq 1/2$:} In this case, when we use Assumption \ref{as:hess_growth} for $\xi \leq 1/2$
\begin{align*}
    \E[\norm{\nabla V(x_{kh})}^4]
    &\lesssim \constL^4 \, (1+ \norm{x_{kh}}^2) \\
    &\lesssim \constL^4\ \bigl(1+\E[\norm{x_0}^2] + (\constB+d) \,kh \bigr)\,.
\end{align*}
As we shall see, it will suffice for simplicity in both cases to use the worst case bound for all $k \leq N$,
\begin{align*}
    \E[\norm{\nabla V(x_{kh})}^4] \lesssim
    \constL^4 \, (\constB + \sigma d)^{3}\, Nh\,.
\end{align*}
Substituting all of these terms into Lemma \ref{lem:mou_disc_bound}, we get for $t \in [kh, (k+1)h]$ and $h \lesssim \frac{1}{L}$
\begin{align*}
    \E\bigl[\norm{\nabla V(x_t) - \E[\nabla V(x_{kh})\mid x_t]}^2\bigr]
    & \lesssim L^2 h^2 \FI(\mu_t \mmid \pi) + L^3 d^2 h^2 + L^2 M^2 d^2 h^4  \\
    &\ \qquad + M^2 \constL^4 \, (\constB +\sigma d)^{3}\,Nh^5 + M^2 d^2 h^2\,.
\end{align*}
Finally, from the differential inequality of Lemma~\ref{lem:fokker_planck}, we get
\begin{align*}
    \partial_t \KL(\mu_t \mmid \pi)
    &\le - \frac{3}{4} \FI(\mu_t \mmid \pi) + \E\bigl[\norm{\nabla V(x_t) - \E[\nabla V(x_{kh}) \mid x_t]}^2\bigr]
\end{align*}
and so for $h \lesssim \frac{1}{L}$,
\begin{align*}
    \FI(\mu_t \mmid \pi)
    &\lesssim - \partial_t \KL(\mu_t\mmid\pi) + (1 \vee L^3 \vee M^2) \,d^2 h^2
    + M^2 \constL^4\,  (\constB +\sigma d)^{3}\,Nh^5\,.
\end{align*}
Finally, we integrate and average over the time horizon to get
\begin{align*}
    \frac{1}{Nh} \int_0^{Nh} \FI(\mu_t \mmid \pi) \, \D t &\lesssim \frac{\KL(\mu_0\mmid \pi)}{Nh} + (1 \vee L^3 \vee M^2) \,d^2 h^2
    + M^2 \constL^4\,  (\constB +\sigma d)^{3}\,Nh^5\,.
\end{align*}
Consequently, if we define $\kappa = 1 \vee L \vee M^{2/3} \vee (M^{1/3} m^{2/3})$, then if $h \asymp \frac{K_0^{1/3}}{\kappa \, {(\constB + \sigma d)}^{2/3} \, N^{1/3}}$, we get
\begin{align*}
    \frac{1}{Nh} \int_0^{Nh} \FI(\mu_t \mmid \pi) \, \D t &\lesssim
     \Bigl( {(b+\sigma d)}^{2/3} K_0^{2/3} + \frac{K_0^{5/3}}{{(b+\sigma d)}^{1/3}} \Bigr)\, \frac{\kappa}{N^{2/3}}\,.
\end{align*}
This completes the proof.
\end{proof}
\subsection{Stochastic gradient setting}

\begin{proof}[Proof of Theorem~\ref{thm:stochastic_setting}]
Using Lemma~\ref{lem:fokker_planck}, we have
\begin{align*}
    \partial_t \KL(\mu_t\mmid\pi)
    &\le -\frac{3}{4} \FI(\mu_t\mmid \pi) + \E[\norm{\nabla V(x_t) - G(x_{kh},\zeta_k)}^2]\,.
\end{align*}
The error term can be bounded via
\begin{align*}
    \E[\norm{\nabla V(x_t) - G(x_{kh},\zeta_k)}^2]
    &\le 3\E[\norm{\nabla V(x_t) - \nabla\hat V(x_t)}^2] + 3\E[\norm{\nabla\hat V(x_t) - \nabla\hat V(x_{kh})}^2] \\
    &\qquad{} + 3\E[\norm{\nabla\hat V(x_{kh}) - G(x_{kh},\zeta_k)}^2] \\
    &\le 3 \, \Cbias + 3 \, \Cvar + 3\hat L^2 \E[\norm{x_t - x_{kh}}^2]\,.
\end{align*}
Next, we have
\begin{align*}
    \E[\norm{x_t - x_{kh}}^2]
    &= {(t-kh)}^2 \E[\norm{G(x_{kh},\zeta_k)}^2] + 2\E[\norm{B_t - B_{kh}}^2] \\
    &\le 2 \, \Cvar \, {(t-kh)}^2 + 2 \, {(t-kh)}^2 \E[\norm{\nabla \hat V(x_{kh})}^2] + 2d \, (t-kh)\,.
\end{align*}
Using smoothness of $\hat V$,
\begin{align*}
    \E[\norm{\nabla \hat V(x_{kh})}^2]
    &\le 2\E[\norm{\nabla \hat V(x_t)}^2] + 2\hat L^2 \E[\norm{x_t - x_{kh}}^2]\,.
\end{align*}
Substitute this into the previous inequality.
If $h \le 1/(\sqrt 8 \hat L)$, we can rearrange to obtain
\begin{align*}
    \E[\norm{x_t - x_{kh}}^2]
    &\le 4 \, \Cvar \, {(t-kh)}^2 + 8 \, {(t-kh)}^2 \E[\norm{\nabla \hat V(x_t)}^2] + 4d \, (t-kh)\,.
\end{align*}

Next, to bound $\E[\norm{\nabla \hat V(x_t)}^2]$, we generalize the proof of Lemma~\ref{lem:fisher_info_arg}.
Introduce the generator $\eu L$ of the Langevin diffusion. Since $\eu L\hat V = -\Delta \hat V + \langle \nabla V, \nabla \hat V\rangle$, we can write
\begin{align*}
    \E_{\mu_t}[\norm{\nabla \hat V}^2]
    &= \E_{\mu_t}[\eu L\hat V + \Delta \hat V + \langle \nabla \hat V, \nabla \hat V - \nabla V\rangle] \\
    &\le \E_{\mu_t} \eu L \hat V + \hat L d + \sqrt{\E_{\mu_t}[\norm{\nabla \hat V}^2] \E_{\mu_t}[\norm{\nabla \hat V - \nabla V}^2]} \\
    &\le \E_{\mu_t} \eu L \hat V + \hat L d + \sqrt{\Cbias \E_{\mu_t}[\norm{\nabla \hat V}^2]}\,.
\end{align*}
For the first term, we can use an integration by parts argument as in the proof of Lemma~\ref{lem:fisher_info_arg}:
\begin{align*}
    \E_{\mu_t} \eu L \hat V
    &= \E_{\mu_t}\bigl\langle \nabla \hat V, \nabla \ln \frac{\mu_t}{\pi}\bigr\rangle
    \le \sqrt{\E_{\mu_t}[\norm{\nabla \hat V}^2] \FI(\mu_t \mmid \pi)}\,.
\end{align*}
Applying Young's inequality,
\begin{align*}
    \E_{\mu_t}[\norm{\nabla \hat V}^2]
    &\le \frac{1}{4}\E_{\mu_t}[\norm{\nabla \hat V}^2] + \FI(\mu_t \mmid \pi) + \hat L d + \frac{1}{4}\E_{\mu_t}[\norm{\nabla \hat V}^2] + \Cbias
\end{align*}
which is rearranged to yield
\begin{align*}
    \E_{\mu_t}[\norm{\nabla \hat V}^2]
    &\le 2\FI(\mu_t \mmid \pi) + 2\hat L d + 2\,\Cbias\,.
\end{align*}

Therefore,
\begin{align*}
    \E[\norm{\nabla V(x_t) - G(x_{kh},\zeta_k)}^2]
    &\le 3\,\Cbias + 3 \, \Cvar + 12\hat L^2 \,\Cvar \, {(t-kh)}^2 \\
    &\qquad{} + 48\hat L^2 \, \{\FI(\mu_t\mmid\pi) + \hat L d + \Cbias\} \,  {(t-kh)}^2 + 12\hat L^2 d \, (t-kh)
\end{align*}
and for $h \le 1/(\sqrt{192}\hat L)$ we can absorb the Fisher information term into the differential inequality for the KL divergence:
\begin{align*}
    \partial_t \KL(\mu_t\mmid\pi)
    &\le -\frac{1}{2} \FI(\mu_t\mmid\pi) + 3\,\Cbias + 3 \, \Cvar + 12\hat L^2 \,\Cvar \, {(t-kh)}^2 \\
    &\qquad{} + 48\hat L^2 \, (\hat L d + \Cbias) \,  {(t-kh)}^2 + 12\hat L^2 d \, (t-kh)\,.
\end{align*}
Integrating,
\begin{align*}
    \KL(\mu_{(k+1)h} \mmid \pi) - \KL(\mu_{kh}\mmid\pi)
    &\le -\frac{1}{2} \int_{kh}^{(k+1)h} \FI(\mu_t\mmid\pi) \, \D t + 3h\,\Cbias + 3h \, \Cvar + 4\hat L^2h^3 \,\Cvar \\
    &\qquad{} + 16\hat L^2 h^3 \, (\hat L d + \Cbias) + 6\hat L^2 dh^2 \\
    &\le -\frac{1}{2} \int_{kh}^{(k+1)h} \FI(\mu_t\mmid\pi) \, \D t + 4h\,\Cbias + 4h \, \Cvar + 8\hat L^2 dh^2\,.
\end{align*}
The proof is concluded in the same way as Theorem~\ref{thm:main}.
\end{proof}

\subsection{Gaussian smoothing}

\begin{proof}[Proof of Lemma~\ref{lem:gaussian_smoothing_bias}]
    From a Gaussian integration by parts argument~\cite{nesterov2017random}, writing $\gamma$ for the standard Gaussian measure on $\R^d$,
    \begin{align*}
        \norm{\nabla \hat V(x) - \nabla V(x)}
        &= \Bigl\lVert \int \Bigl( \frac{V(x+\eta\zeta) - V(x)}{\eta} - \langle \nabla V(x), \zeta\rangle \Bigr) \, \zeta \, \gamma(\D \zeta)\Bigr\rVert \\
        &\le \frac{1}{\eta} \int \abs{V(x+\eta\zeta) - V(x) - \eta \, \langle \nabla V(x), \zeta\rangle} \, \norm\zeta \, \gamma(\D\zeta) \\
        &= \int \Bigl\lvert \Bigl\langle \int_0^1 \{\nabla V(x+t\eta\zeta) - \nabla V(x)\} \,\D t, \zeta\Bigr\rangle\Bigr\rvert \, \norm\zeta \, \gamma(\D\zeta) \\
        &\le \int \Bigl(\int_0^1 \norm{\nabla V(x+t\eta\zeta) - \nabla V(x)} \,\D t \Bigr) \, \norm\zeta^2 \, \gamma(\D\zeta) \\
        &\le L\eta^s \int \norm \zeta^{2+s} \, \gamma(\D \zeta)
        \asymp Ld^{(2+s)/2}\eta^s\,.
    \end{align*}
    The last inequality follows from standard bounds on the Gaussian moments.
\end{proof}

\begin{proof}[Proof of Corollary~\ref{cor:poincare_weakly_smooth}]
    We proceed via the following steps.

    \underline{\textbf{1. Control of the bias.}} Let $\hat\pi \propto\exp(-\hat V)$ and assume that the potential $V$ is normalized so that $\int\exp(-V) = 1$. From~\cite[Lemma 2.2]{chatterji2020langevin}, we know that $\sup{\abs{\hat V - V}} \le Ld^{(1+s)/2} \eta^{1+s}$.
    Then,
    \begin{align*}
        \frac{\hat\pi}{\pi}
        &= \frac{\exp(V - \hat V)}{\int \exp(-\hat V)}
        \le \frac{\exp(V - \hat V)}{\exp(-\sup{\abs{V - \hat V}})\int \exp(-V)}
        \le \exp\bigl(2 \sup{\abs{V - \hat V}}\bigr)\,.
    \end{align*}
    For $\eta$ small, we deduce from Pinsker's inequality that
    \begin{align*}
        \norm{\hat\pi-\pi}_{\rm TV}^2
        &\lesssim \KL(\hat \pi \mmid \pi)
        \le \ln \sup \frac{\hat\pi}{\pi}
        \le 2\sup{\abs{V - \hat V}}
        \lesssim Ld^{(1+s)/2} \eta^{1+s}\,.
    \end{align*}
    Hence, provided $\eta \lesssim \varepsilon^{1/(1+s)}/(L^{1/(1+s)} d^{1/2})$, we can ensure that $\norm{\hat \pi - \pi}_{\rm TV}^2 \le \frac{\varepsilon}{4}$.
    
    \underline{\textbf{2. Convergence to the smoothed potential.}} We next apply Theorem~\ref{thm:stochastic_setting} with the target distribution $\hat\pi$. Due to the mini-batching of the stochastic gradients,
    \begin{align*}
        \hat L \lesssim \frac{L d^{(1-s)/2}}{\eta^{1-s}}\,,\qquad \Cvar\lesssim \frac{L^2 d^s \eta^{2s}}{B}\,.
    \end{align*}
    Since we are viewing the smoothed potential $\hat\pi$ as the target, then $\Cbias = 0$. Therefore, Theorem~\ref{thm:stochastic_setting} implies that $\FI(\bar\mu_{Nh}\mmid \hat\pi) \le \delta$ after $N$ iterations, provided that $\eta \lesssim B^{1/(2s)} \delta^{1/(2s)}/(L^{1/s} d^{1/2})$ and
    \begin{align*}
        N \gtrsim \frac{K_0 L^2 d^{2-s}}{\delta^2 \eta^{2\,(1-s)}}\,.
    \end{align*}
    
    \underline{\textbf{3. The smoothed potential satisfies a Poincar\'e inequality.}} From the first step, our choice of $\eta$ entails that $\hat\pi$ is a bounded perturbation of $\pi$, and hence $\hat\pi$ satisfies a Poincar\'e inequality with constant $\lesssim C_{\msf{PI}}$~\cite[Proposition 4.2.7]{bakrygentilledoux2014}.
    Applying Lemma~\ref{lem:transport_ineq}, we obtain
    \begin{align*}
        \norm{\bar\mu_{Nh} - \hat\pi}_{\rm TV}^2
        &\lesssim C_{\msf{PI}} \FI(\bar\mu_{Nh} \mmid \hat\pi)\,.
    \end{align*}
    Setting $\delta \asymp \varepsilon/C_{\msf{PI}}$, we see that provided $\eta \lesssim B^{1/(2s)} \varepsilon^{1/(2s)}/(C_{\msf{PI}}^{1/(2s)} L^{1/s} d^{1/2})$ and
    \begin{align*}
        N \gtrsim \frac{C_{\msf{PI}}^2 K_0 L^2 d^{2-s}}{\varepsilon^2 \eta^{2\,(1-s)}}\,,
    \end{align*}
    we obtain $\norm{\bar\mu_{Nh} - \hat\pi}_{\rm TV}^2 \le \frac{\varepsilon}{4}$.
    
    \underline{\textbf{4. Conclusion of the proof.}} Putting the steps together,
    \begin{align*}
        \norm{\bar\mu_{Nh} - \pi}_{\rm TV}^2
        &\le 2 \, \norm{\bar\mu_{Nh} - \hat\pi}_{\rm TV}^2 + 2 \, \norm{\hat\pi-\pi}_{\rm TV}^2
        \le \varepsilon\,.
    \end{align*}
    To fulfill the conditions on $\eta$, we take
    \begin{align}\label{eq:batch_optimal_step_size}
        \eta \asymp \frac{1}{d^{1/2}} \min\Bigl\{ \frac{\varepsilon^{1/(1+s)}}{L^{1/(1+s)}}, \frac{B^{1/(2s)} \varepsilon^{1/(2s)}}{C_{\msf{PI}}^{1/(2s)} L^{1/s}}\Bigr\}\,.
    \end{align}
    The gradient complexity is
    \begin{align*}
        BN
        &\asymp \frac{C_{\msf{PI}}^2 K_0 L^2 d^{3-2s}}{\varepsilon^2} \times B \times \max\Bigl\{ \frac{L^{1/(1+s)}}{\varepsilon^{1/(1+s)}}, \frac{C_{\msf{PI}}^{1/(2s)} L^{1/s}}{B^{1/(2s)} \varepsilon^{1/(2s)}}\Bigr\}^{2 \, (1-s)}\,.
    \end{align*}
    Now we optimize over $B$.
    If $s \ge 1/2$, then we set $B = 1$, with complexity
    \begin{align*}
        BN
        &\asymp \frac{C_{\msf{PI}}^{(1+s)/s} K_0 L^{2/s} d^{3-2s}}{\varepsilon^{(1+s)/s}}\,.
    \end{align*}
    Otherwise, if $s\le 1/2$, we set $B\asymp C_{\msf{PI}} L^{2/(1+s)}/\varepsilon^{(1-s)/(1+s)}$, with complexity
    \begin{align*}
        BN
        &\asymp \frac{C_{\msf{PI}}^3 K_0 L^{6/(1+s)} d^{3-2s}}{\varepsilon^{(5-s)/(1+s)}}\,.
    \end{align*}
    This completes the proof.
\end{proof}

\subsection{Finite sum setting}
\begin{proof}[Proof of Theorem~\ref{thm:minibatch}] Let ${(x_t)}_{t\ge 0}$ denote the interpolation of~\eqref{eq:vr-lmc}. Using Lemma~\ref{lem:fokker_planck}, for $t \in [kh, (k+1)h]$, we have
    \begin{align}
    \label{eq:original1}
        \partial_t \msf{KL}(\mu_t \mmid \pi)
        &\le - \frac{3}{4} \FI(\mu_t\mmid \pi) + \E[\norm{\nabla V(x_t) - g_k}^2] \nonumber\\
        &\le - \frac{3}{4} \FI(\mu_t\mmid \pi) + 2\E[\norm{\nabla V(x_t) - \nabla V(x_{kh})}^2] + 2\E[\norm{\nabla V(x_{kh}) - g_k}^2]\,.
    \end{align}
The second term in~\eqref{eq:original1} can be further bounded as
\begin{align}
\label{eq:lipsch}
    \E[\norm{\nabla V(x_t) - \nabla V(x_{kh})}^2] &\leq L^2\E[\norm{x_t - x_{kh}}^2] = L^2 \,{(t-kh)}^2 \E[\norm{g_k}^2] + 2L^2 d \, (t-kh) \nonumber\\
    &\leq L^2 h^2 \E[\norm{g_k}^2] + 2L^2 dh = L^2\E[\norm{x_{(k+1)h} - x_{kh}}^2] \,.
\end{align}
Furthermore, write $\sigma_k^2 = \E[\norm{g_k - \nabla V(x_{kh})}^2]$  for the variance term. The third term in~\eqref{eq:original1} can be bounded as
    \begin{align*}
        \sigma_{k+1}^2 &= (1-p)\E [\norm{g_{k} - \nabla V(x_{(k+1)h}) + \nabla f_i(x_{(k+1)h}) - \nabla f_i(x_{kh})}^2]\\
        &= (1-p)\E [\lVert g_{k} - \nabla V(x_{kh}) + \underbrace{(\nabla f_i(x_{(k+1)h}) - \nabla f_i(x_{kh}))}_{:= a_i} \\
        &\qquad\qquad\qquad\qquad\qquad\qquad\qquad{} - \underbrace{(\nabla V(x_{(k+1)h}) - \nabla V(x_{kh}))}_{:= \bar{a} = \frac{1}{n}\sum_{\ell=1}^n a_\ell}\rVert^2]\\
        &= (1-p)\E [\norm{g_{k} - \nabla V(x_{kh})}^2] + (1-p)\,\frac{1}{n}\sum_{\ell=1}^n\E [\norm{a_\ell - \bar{a}}^2]\\
        &\leq (1-p)\E [\norm{g_{k} - \nabla V(x_{kh})}^2] + (1-p)\,\frac{1}{n}\sum_{\ell=1}^n\E [\norm{a_\ell}^2]\\
        &\leq (1-p)\,\sigma_k^2 + (1-p)\, L^2\E[\norm{x_{(k+1)h} - x_{kh}}^2]\,.
    \end{align*}
    In the third equality, we conditioned w.r.t.\ $\mathcal F_k$ and used that $i$ is independent of $\mathcal F_k$.
    Therefore, we obtain the inequality
    \begin{equation}
    \label{eq:variance}
        \sigma_k^2 \leq  \frac{1-p}{p} \,L^2\E[\norm{x_{(k+1)h} - x_{kh}}^2] - \frac{1}{p} \,(\sigma_{k+1}^2 - \sigma_{k}^2)\,.
    \end{equation}
Plugging~\eqref{eq:variance} and~\eqref{eq:lipsch} into~\eqref{eq:original1} we obtain    
\begin{align}
\label{eq:iter}
        \partial_t \msf{KL}(\mu_t \mmid \pi)
        \le& - \frac{3}{4} \FI(\mu_t\mmid \pi) + \frac{2L^2}{p}\E[\norm{x_{(k+1)h} - x_{kh}}^2] - \frac{2}{p}\,(\sigma_{k}^2 - \sigma_{k+1}^2)\,.
\end{align}
Now, we bound the term $\mathbb{E}[\norm{x_{(k+1)h} - x_{kh}}^2]$ appearing in~\eqref{eq:iter} as
\begin{align*}
    \E[\norm{x_{(k+1)h} - x_{kh}}^2] &= h^2 \E[\norm{g_k}^2] + 2hd \\
    &\leq h^2 \E[\norm{\nabla V(x_{kh})}^2] + h^2 \sigma_k^2 + 2hd\\
    &\leq 2 h^2 \E[\norm{\nabla V(x_t)}^2] + 2 h^2 \E[\norm{\nabla V(x_t) - \nabla V(x_{kh})}^2] + h^2 \sigma_k^2 + 2hd\\
    &\leq 2 h^2 \E[\norm{\nabla V(x_t)}^2] + 2L^2 h^2 \E[\norm{x_{(k+1)h} - x_{kh}}^2] + h^2 \sigma_k^2 + 2hd\,,
\end{align*}
where we used~\eqref{eq:lipsch}. Further using~\eqref{eq:variance}, we obtain
\begin{align*}
    \E[\norm{x_{(k+1)h} - x_{kh}}^2]
    &\leq 2 h^2 \E[\norm{\nabla V(x_t)}^2] - \frac{h^2}{p}\,(\sigma_{k+1}^2 - \sigma_k^2) \\
    &\qquad{} + h^2L^2\,\frac{1+p}{p}\E[\norm{x_{(k+1)h} - x_{kh}}^2] + 2hd\,.
\end{align*}
Assuming ${h^2 L^2} \le p/24$, we have
\begin{align*}
    \frac{11}{12} \E[\norm{x_{(k+1)h} - x_{kh}}^2] 
    &\leq 2 h^2 \E[\norm{\nabla V(x_t)}^2] - \frac{h^2}{p}\,(\sigma_{k+1}^2 - \sigma_k^2) + 2hd\,.
\end{align*}
Using Lemma~\ref{lem:fisher_info_arg}, we obtain 
\begin{align}
    2\E[\norm{x_{(k+1)h} - x_{kh}}^2] 
    &\leq 6 h^2 \FI(\mu_t\mmid \pi) - \frac{3h^2}{p}\,(\sigma_{k+1}^2 - \sigma_k^2) + 6hd + 12L h^2 d \nonumber \\
    &\le 6 h^2 \FI(\mu_t\mmid \pi) - \frac{3h^2}{p}\,(\sigma_{k+1}^2 - \sigma_k^2) + 9hd\,. \label{eq:temp1}
\end{align}
Plugging~\eqref{eq:temp1} into~\eqref{eq:iter}, we obtain
\begin{align}
        \partial_t \msf{KL}(\mu_t \mmid \pi)
        &\le \bigl(- \frac{3}{4} + \frac{6L^2h^2}{p} \bigr)  \FI(\mu_t\mmid \pi) + \frac{9L^2 h d}{p} - \frac{2}{p}\,\bigl(1+\frac{3L^2 h^2}{2p}\bigr)\,(\sigma_{k+1}^2 - \sigma_{k}^2) \nonumber\\
        &\leq -\frac{1}{2}\FI(\mu_t\mmid \pi) + \frac{9L^2 h d}{p} - \frac{2}{p}\,\bigl(1+\frac{3L^2 h^2}{2p}\bigr)\,(\sigma_{k+1}^2 - \sigma_{k}^2)\,, 
    \end{align}
where we used $L^2 h^2 \le p/24$. Integrating between $kh$ and $(k+1)h$,
\begin{align*}
    \mcL_{k+1} - \mcL_{k} \leq -\frac{1}{2}\int_{kh}^{(k+1)h} \FI(\mu_t\mmid \pi)\, \D t + \frac{9L^2 h d}{p}\,,
\end{align*}
where $\mcL_k := \msf{KL}(\mu_{kh} \mmid \pi) + \frac{2h}{p}\, (1+\frac{3L^2 h^2}{2p}) \,\sigma_k^2 \geq 0$.
Iterating, and using $\mcL_k \geq 0$,
\begin{align*}
    \frac{1}{Nh}\int_{0}^{Nh} \FI(\mu_t\mmid \pi) \, \D t \leq \frac{2\mcL_0}{Nh} + \frac{18L^2 h d}{p}\,.
\end{align*}
Since $h^2 L^2 < p/24$, we have
\begin{align*}
    \mcL_0 = \msf{KL}(\mu_{0} \mmid \pi) + \frac{2h}{p}\,\bigl(1+\frac{3L^2 h^2}{2p}\bigr)\,\sigma_0^2 
    \leq \msf{KL}(\mu_{0} \mmid \pi) + \frac{3h}{p}\,\sigma_0^2 = C\,,
\end{align*}
thereby completing the first claim.  By setting $h = \frac{\sqrt{p C}}{3L \sqrt{Nd}}$, we obtain the second.
\end{proof}

\end{document}